\documentclass[10pt,a4paper]{article}
\usepackage{amssymb,amsmath,amsopn,amsthm,graphicx,makeidx}
\input xy
\xyoption{all}
\usepackage[all,2cell]{xy} \UseAllTwocells

\newtheorem{definition}{Definition}[section]

\newtheorem{prop}[definition]{Proposition}
\newtheorem{theorem}[definition]{Theorem}

\theoremstyle{remark}

\theoremstyle{remark}

\theoremstyle{remark}

\theoremstyle{remark}
\newtheorem{example}[definition]{Example}
\theoremstyle{remark}

\theoremstyle{remark}

\theoremstyle{remark}

\theoremstyle{remark}

\theoremstyle{remark}

\theoremstyle{definition}

\theoremstyle{definition}

\theoremstyle{remark}

\renewcommand{\marginpar}[2][]{}

\renewenvironment{proof}{\noindent {\bf{Proof.}}}{\hspace*{3mm}{$\Box$}{\vspace{9pt}}}

\begin{document}

\title{Modules as exact functors}
\author{Mike Prest\footnote{The author acknowledges the support of EPSRC grant EP/K022490/1 during the development of this work.},\\ Alan Turing Building
\\School of Mathematics\\University of Manchester\\
Manchester M13 9PL\\UK\\mprest@manchester.ac.uk}

\maketitle

\tableofcontents

\section{Introduction}

This mostly expository paper is about modules, by which I mean the following.

\vspace{4pt}

\noindent A {\bf module} is an exact functor from a small abelian category to the category, ${\bf Ab}$, of abelian groups.

\vspace{4pt}

An {\bf exact functor} is, recall, one which takes exact sequences to exact sequences.

Certainly, this is not the usual definition of a module.  But it is a good alternative, in some senses better.  And it gives a different perspective on modules.  I will explain how this definition arises and we will see where it leads us.  Of course, I should, and will, show that this definition is equivalent to the usual one.

The first step towards this definition of module is into a functor-category-theoretic framework, where we allow rings to be multi-sorted and where modules are just additive functors.  One of the aims of the paper is to show that this all is very concrete and natural, especially through explicit computation of examples.  Another is to show how model theory gives a convenient language for recognising and handling the multi-sorted structures that arise in this framework.  Yet another aim is to reveal implicit definable structure and show how to make this explicit, indeed, that is really the main step in reaching the rather strange-looking definition of ``module" above.

I give most definitions, but the paper is really aimed at a reader who has at least a little background on rings, modules, quivers and representations and has seen some very basic category theory.

The ``narrative" part of the paper proceeds as follows:  modules as additive functors; the free abelian category generated by a ring (as a functor category and {\it via} systems of linear equations); the abelian category associated to a module (a quotient =  localisation of a free abelian category); definable categories (additive categories which include, but are more general than, module categories) and the definable category generated by a module.

Most of the rest of the paper is filling-in/linking/examples/context, as well as explanation of concepts which are less likely to be familiar to all readers:  Serre-localisation; many-sorted structures; definable structure/imaginary sorts; the idea of interpretation from model theory; interpretations regarded as functors; pp formulas; languages for modules.  It is my hope that the reader will be able to dip in and out of the various sections and subsections.\footnote{Thanks to Mike Bushell, Lorna Gregory, Harry Gulliver and Sam Dean for comments on a preliminary version of this paper.}

\vspace{4pt}

I'll start with an example which illustrates many of the main points and which we will gradually develop.

\subsection{Finding modules within modules: an example}\label{secKT}\marginpar{secKT}  This is an example of how essentially the same module structure may be found over very different rings.  By ``essentially the same" I mean something which includes, but is much more general than, Morita equivalence between categories of modules.

Consider the quiver $\widetilde{A_1}$:  $\xymatrix{1 \ar@/^/[r]^\alpha \ar@/_/[r]_\beta & 2}$.  Given a field $K$, a $K$-representation of $\widetilde{A_1}$ is given by two vector spaces $V_1$ and $V_2$ (one for each vertex of the quiver) and two $K$-linear maps $t_\alpha$, $t_\beta$ from $V_1$ to $V_2$ (corresponding to the arrows of the quiver)  $\xymatrix{V_1 \ar@/^/[r]^{t_\alpha} \ar@/_/[r]_{t_\beta} & V_2}$.  One way to build such representations is to take a $K[T]$-module $M$ - that is a $K$-vectorspace with a distinguished endomorphism (the action of $T$) and, from it, define the representation with $V_1=V_2=M$, $t_\alpha = 1_M$ and $t_\beta = T$ (meaning multiplication-by-$T$)  $\xymatrix{M \ar@/^/[r]^1 \ar@/_/[r]_T & M}.$ Note that $t_\alpha$ is an isomorphism.

Recall that the category of $K$-representations of a quiver, such as $\widetilde{A_1}$, is equivalent to the category of modules over the {\bf path} algebra $K\widetilde{A_1}$.  This algebra has a $K$-vectorspace basis consisting of all the paths in the quiver (including a lazy path at each vertex), in this case $e_1, e_2, \alpha, \beta$.  The multiplication is determined by defining the product of two paths to be their composition, written right to left, if defined, otherwise 0.  Thus, for example, $e_1^2=e_1$, $\alpha e_1= \alpha = e_2\alpha$.

From a representation $V$ of $\widetilde{A_1}$ we define the module over $K\widetilde{A_1}$ which has underlying set $V_1\oplus V_2$ and with actions given by $e_1(v_1,v_2)=(v_1,0)$, $e_2(v_1, v_2) =  (0,v_2)$, $\alpha (v_1, v_2) = (0,t_\alpha v_1)$, $\beta (v_1, v_2) = (0, t_\beta v_1)$.

Conversely, from a $K\widetilde{A_1}$-module $N$ we define the representation with vector spaces $e_1N$, $e_2N$ (noting that $N=e_1N\oplus e_2N$ as $K$-vectorspaces) and with maps $t_\alpha$, $t_\beta$ given by $\alpha$ and $\beta$ composed with the relevant injections and projections from and to these components. Combining with the process from the first paragraph, we have a map from $K[T]$-modules to $K\widetilde{A_1}$-modules given by taking $M$ to $M\oplus M$ with the induced action of the path algebra $K\widetilde{A_1}$.  \\ \begin{center} $K[T]$-modules $\rightarrow$ $K$-representations of $\widetilde{A_1}$ $\rightarrow$ $K\widetilde{A_1}$-modules.\end{center}

In the other direction, suppose we are given a $K\widetilde{A_1}$-module $N$ where $\alpha$ restricts to a bijection from $e_1N$ to $e_2N$, equivalently a representation $V$ of $\widetilde{A_1}$ in which $t_\alpha$ is an isomorphism.  Then the vector space $e_1N$, equipped with the endomorphism defined in terms of representations as $t_\alpha^{-1}t_\beta$, is a $K[T]$-module from which the original $K\widetilde{A_1}$-module may be recovered.

If one of these processes is applied and then the other, then the result is isomorphic to the module we began with.  These processes extend to morphisms and are functorial.  They give us an equivalence between the category of $K[T]$-modules and a large subcategory of the category of $K\widetilde{A_1}$-modules.

This is an equivalence which, like Morita equivalence, is very concrete and reversible:  we start with a module, we construct a new module over a different ring and there is a construction in the other direction which recovers the original module.  So, in some sense, the two modules - one over $K[T]$, the other over $K\widetilde{A_1}$ - are ``implicitly the same" structure.  We will make this explicit and precise.  In particular, these will turn out to give isomorphic modules in the sense defined at the start of this paper.

\section{Modules} \label{secmods}\marginpar{secmods}

By ``the usual definition" of a module I mean: a {\bf module} is an additive functor from a skeletally small preadditive category to ${\bf Ab}$.  Perhaps that is stretching the meaning of ``usual" since it is a generalisation, but a mild one (\cite{Mit}), of the first definition one meets - the case where the preadditive category has just one object.  Let us start there, with a ring $R$ which has an identity $1$.

\vspace{4pt}

\noindent {\bf Rings as categories:}   We view $R$ as a category which has a single object, $\ast_R$ say, with the elements of $R$ as the arrows from $\ast_R$ to $\ast_R$.  The addition in $R$ gives an additive structure on this category, that is, the addition on the set $(\ast_R, \ast_R)$ of arrows, and the multiplication in $R$ gives composition of morphisms (to the left).  Thus we regard $R$ as a {\bf preadditive category} - a category where each hom set has an abelian group structure and where composition is bilinear.

\begin{example}\label{KT0}\marginpar{KT0}  Let $K$ be a field and take $R$ to be the ring $K[T]$ of polynomials in one variable with coefficients from $K$.  In the corresponding 1-object preadditive category, $1\in K[T]$ is the identity morphism ${\rm id}_{\ast_{K[T]}}$, $T$ is an endomorphism of $\ast_{K[T]}$,  $T^2$ is the composition of $T$ with itself, indeed every morphism in $(\ast_{K[T]}, \ast_{K[T]})$ is a $K$-linear combination of $1$ and powers of $T$.  Note that, since $K[T]$ is a $K$-algebra, $(\ast_{K[T]}, \ast_{K[T]})$ is actually a $K$-vector space.
\end{example}

\vspace{4pt}

\noindent {\bf Modules as functors:}  An additive functor\footnote{A functor $F$ between preadditive categories is {\bf additive} if $F(f+g) =Ff+Fg$ whenever $f$ and $g$ have the same domain and the same codomain.} $M$ from $R$, regarded as a 1-object category, to ${\bf Ab}$ is given by an abelian group $M(\ast_R)$ and, for each $r\in R$, an endomorphism, which we call multiplication by $r$, of $M(\ast_R)$.  One may check that the condition that $M$ is an additive functor translates to this data defining a left $R$-module structure on $M(\ast_R)$.  Moreover, a natural transformation from the functor $M$ to another additive functor $N$ is a morphism of abelian groups $M(\ast_R) \rightarrow N(\ast_R)$ which commutes with multiplication by $r$ for every $r\in R$ - that is, it is an $R$-linear map from $M$ to $N$.  In this way, the category $(R,{\bf Ab})$ of (covariant) additive functors from $R$ to ${\bf Ab}$ is precisely the category $R\mbox{-}{\rm Mod}$ of left $R$-modules.  Similarly, the category ${\rm Mod}\mbox{-}R$ of right $R$-modules is the category $(R^{\rm op}, {\bf Ab})$ of contravariant functors. (All functors in this paper will be additive, so let us drop that adjective, except when used for emphasis, from now on.)

\begin{example} \label{KT0m}\marginpar{KT0m}
Continuing Example \ref{KT0}, to give a functor $M$ from $K[T]$, considered as a category, to ${\bf Ab}$, it is enough to specify the abelian group $M(\ast_{K[T]})$, the $K$-vectorspace structure on $M(\ast_{K[T]})$ given by the scalar multiplications $M(\lambda)$ for $\lambda \in k$ and the action, $M(T)$, of $T$ on $M(\ast_{K[T]})$.  Which is exactly the data of a $K[T]$-module.
\end{example}

A general preadditive category ${\cal R}$ will have more objects but we assume that, up to isomorphism, there is just a set of these - that is, we assume that ${\cal R}$ is {\bf skeletally small} (${\cal R}$ is {\bf small} if it has just a set of objects).  As in the 1-object case, an (additive) functor from ${\cal R}$ to ${\bf Ab}$ may be referred to as a (left) ${\cal R}$-module.  That is, ${\cal R}\mbox{-}{\rm Mod} \simeq ({\cal R}, {\bf Ab})$ and ${\rm Mod}\mbox{-}{\cal R} \simeq ({\cal R}^{\rm op}, {\bf Ab})$.

\begin{example}\label{A20}\marginpar{A20}  Let $K$ be a field and consider the preadditive (in fact, $K$-linear, see below) category ${\cal R}$ generated by two objects $\ast_1, \ast_2$ and a single arrow $\alpha:\ast_1 \rightarrow \ast_2$ from one to the other.  Thus, $(\ast_1,\ast_1)=K.{\rm id}_{\ast_1}$, $(\ast_1,\ast_2) =K.\alpha$ and $(\ast_2,\ast_2)=K.{\rm id}_{\ast_2}$.  This is the path category (see below) of the quiver $A_2$ which is the directed graph $1 \xrightarrow{\alpha} 2$.

An ${\cal R}$-module $M$ is given by two abelian groups (indeed, as in \ref{KT0}, these will be $K$-vector spaces) $M(\ast_1)$, $M(\ast_2)$ and a ($K$-)linear map $M(\alpha): M(\ast_1) \rightarrow M(\ast_2)$.  That is, an ${\cal R}$-module is a $K$-representation of the quiver $A_2$.  And a natural transformation from one module $M$ to another $N$ is precisely a morphism between the corresponding representations, that is, a pair of maps $f_1:M(\ast_1) \rightarrow N(\ast_1)$, $f_2:M(\ast_2) \rightarrow N(\ast_2)$,  such that, for every $\lambda\in k$, the following diagram commutes.

$\xymatrix{M(\ast_1) \ar[r]^{M(\lambda \alpha)} \ar[d]^{f_1} & M(\ast_2) \ar[d]^{f_2} \\
N(\ast_1) \ar[r]^{N(\lambda \alpha)} & N(\ast_2) }$

In this paper we will use the concept of sorts from model theory, see Section \ref{secmodth}; this is well-illustrated in this example since any ${\cal R}$-module $M$ is naturally a 2-sorted structure, the sorts being $M(\ast_1)$ and $M(\ast_2)$.  Replacing $M$ by the corresponding module over the path algebra would be to replace this two-sorted structure by a one-sorted structure (based on the single sort $M(\ast_1)\oplus M(\ast_2)$).
\end{example}

Like any preadditive category ${\cal R}$ with only finitely many objects (up to isomorphism), the category in the example above is essentially equivalent to a ring $R$ (namely the path algebra) in the sense that the functors from ${\cal R}$ to ${\bf Ab}$ are equivalent to the $R$-modules as usually defined.  But preadditive categories include examples such as the category ${\cal R} = R\mbox{-}{\rm mod}$ of finitely presented modules over any ring (or small preadditive category) $R$.  Then ${\cal R}\mbox{-}{\rm Mod}$, the category $(R\mbox{-}{\rm mod}, {\bf Ab})$ of functors from $R\mbox{-}{\rm mod}$ to ${\bf Ab}$, is almost never equivalent to the category of modules over any 1-object ring (exceptions being von Neumann regular rings and rings of finite representation type).

\vspace{4pt}

\noindent {\bf Functor categories:}  If ${\cal C}$, ${\cal D}$ are preadditive categories with ${\cal C}$ skeletally small (to avoid set-theoretic issues), we use the notation $({\cal C}, {\cal D})$ for the category of additive functors from ${\cal C}$ to ${\cal D}$.  The objects of this category are the functors; the arrows are the natural transformations between functors.  This is itself a preadditive category:  if $\sigma, \tau:F\rightarrow G$ are natural transformations from $F$ to $G$, where $F,G:{\cal C} \rightarrow {\cal D}$ are additive functors, then we define $\sigma + \tau: F\rightarrow G$ to have, for its component $(\sigma+\tau)_C$ at $C\in {\cal C}$, the sum $\sigma_C + \tau_C$ of the components of $\sigma$ and $\tau$ at $C$.  We will always use this notation with ${\cal D}$ an {\bf additive} category, meaning that ${\cal D}$ has a zero object as well as all finite direct sums of objects; in which case $({\cal C}, {\cal D})$ also will be additive.

\begin{example} If ${\cal T}$ is a compactly generated triangulated category then, setting ${\cal T}^{\rm c}$ to be the full subcategory of compact objects (this is skeletally small and additive), we may consider the functor category ${\rm Mod}\mbox{-}{\cal T}^{\rm c} = (({\cal T}^{\rm c})^{\rm op}, {\bf Ab})$ and the restricted Yoneda functor ${\cal T} \rightarrow {\rm Mod}\mbox{-}{\cal T}^{\rm c}$ given by $T\mapsto (-,T)\upharpoonright_{{\cal T}^{\rm c}}$ on objects.  This is an abelian (Grothendieck) approximation of ${\cal T}$ and using it allows techniques and results from abelian categories to be applied in this non-abelian context (see, e.g., \cite{KraTel}, \cite{ALPP}).
\end{example}

\noindent {\bf $K$-linear categories:}  A general point seen in examples above is that, if $K$ is a field, or just a commutative ring, and ${\cal R}$ is a $K$-{\bf linear} category, meaning that each morphism group $(a,b)$ is a $K$-module with composition being $(K,K)$-bilinear, then every additive functor from ${\cal R}$ to ${\bf Ab}$ is $K$-linear.  That is, directly from the definition of a functor, each group $M(a)$ has an induced $K$-module structure and each map $M(f): M(a) \rightarrow M(b)$, where $f\in (a,b)$ is $K$-linear.  Since every additive functor to ${\bf Ab}$ will automatically be a $K$-linear map to a $K$-module, the category $({\cal R}, {\bf Ab})$ is, therefore, naturally equivalent to the category $({\cal R}, K\mbox{-}{\rm Mod})$ of $K$-linear functors to $K$-modules and so, in the definition of ``${\cal R}$-module'', we may take the codomain category to be $K\mbox{-}{\rm Mod}$.

\vspace{4pt}

\noindent {\bf Path categories:} Generalising \ref{A20}, let $Q$ be a quiver and let $K$ be a commutative ring.  We define the $K$-{\bf path category} $K\overrightarrow{Q}$ of $Q$ to have, for objects, the vertices of $Q$ and, given vertices $s,t$ of $Q$, the set $(s,t)$ of arrows is defined to be the free $K$-module on the set of {\bf paths} (concatenations of arrows) from $s$ to $t$.  Composition is given by concatenation of paths when the end of the one path is the start of the next, $0$ when they don't connect, as in the usual definition of the path algebra of a quiver.  This path category is a small preadditive category, the modules over which are the $K$-representations of $Q$.  If the number of vertices of $Q$ is finite then $K\overrightarrow{Q}$ is morally equivalent to the path algebra $KQ$, which is a ring in the usual (one-object) sense, though these are not literally equivalent as categories unless $Q$ has only one vertex.  What are equivalent are their {\bf idempotent-splitting}={\bf karoubian} additive completions.  To get that, we add a $0$ object and form finite direct sums of objects (that gives us the additive category generated by the initial preadditive category), then we add kernels and cokernels of idempotent endomorphisms.

One may check that none of these operations (additive completion, idempotent-splitting completion) on a preadditive category changes its category of modules, at least, not up to equivalence.

\begin{example}  (additive and idempotent-splitting completion) For instance, if our initial preadditive category is a ring $R$, regarded as a 1-object category, then the additive category we form from this has finite (including empty) formal powers of the single object $\ast_R$ of $R$ for its objects.  The morphisms from $(\ast_R)^n$ to $(\ast_R)^m$ will be the $m\times n$ matrices of maps from $\ast_R$ to $\ast_R$, that is $m\times n$ matrices over $R$ (the empty power is the zero object, with only the zero morphisms to and from it).  This is equivalent, through the Yoneda embedding, to the opposite of the category of finitely generated free $R$-modules.  On the other hand, the idempotent-splitting completion is equivalent to the category of finitely generated projective modules.
\end{example}

\begin{example}\label{A3} Consider the quiver $A_3$:  $1 \xrightarrow{\alpha} 2 \xrightarrow{\beta} 3$.  The corresponding $K$-path category has the three objects, labelled say 1, 2, 3, and 1-dimensional morphism spaces between each pair $(i,j)$ of vertices with $i\leq j$, has basis the identity (if the vertices are equal), $\alpha$ (from 1 to 2), $\beta$ (from 2 to 3) and $\beta\alpha$ (from 1 to 3).  All other morphism spaces are zero.  The corresponding path algebra can be represented as the ring of $3\times 3$ lower-triangular matrices over $K$.
\end{example}

\begin{example}\label{A1tilde} In the case of the quiver $\widetilde{A_1}$ $\xymatrix{1 \ar@/^/[r]^\alpha \ar@/_/[r]_\beta & 2}$ the hom space from 1 to 2 is 2-dimensional.  The path algebra can be represented as $\left( \begin{array}{cc} Ke_1 &  0 \\ K\alpha \oplus K\beta & Ke_2 \end{array} \right)$.
\end{example}

\section{The free abelian category}\label{secfreeab}\marginpar{secfreeab}

In what way, then, is a module in the sense defined in Section \ref{secmods} equivalent to an exact functor on a small abelian category?  The equivalence uses Freyd's free abelian category.

\begin{theorem}\label{freeabcat}\marginpar{freeabcat} (\cite[4.1]{FreydLJ}) Let ${\cal R}$ be any skeletally small preadditive category.  There is a full additive embedding $j$ of ${\cal R}$ into a skeletally small abelian category ${\rm Ab}({\cal R})$ defined, up to equivalence, by the property that, given any additive functor $M$ from ${\cal R}$ to ${\bf Ab}$, there is a unique-to-equivalence extension through $j$ to an exact functor $\widetilde{M}: {\rm Ab}({\cal R}) \rightarrow {\bf Ab}$.

$\xymatrix{{\cal R} \ar[d]_M \ar[r]^j & {\rm Ab}({\cal R}) \ar[dl]^{\widetilde{M}} \\ {\bf Ab}}$
\end{theorem}

We refer to ${\rm Ab}({\cal R})$ as the {\bf free abelian category on} ${\cal R}$.

Before continuing, we recall that if $G:{\cal A} \rightarrow {\cal B}$ is an exact functor between abelian categories then the kernel of $G$, ${\cal S} = \{ A\in {\cal A}: GA=0\}$, is a {\bf Serre subcategory} of ${\cal A}$:  that is, if $0\rightarrow A \rightarrow B\rightarrow C \rightarrow 0$ is an exact sequence in ${\cal A}$ then $B\in {\cal S}$ iff $A, C\in {\cal S}$.  Conversely, if ${\cal S}$ is a Serre subcategory of ${\cal A}$ then there is a {\bf quotient category} ${\cal A}/{\cal S}$, which is abelian, and an exact functor ${\cal A} \rightarrow {\cal A}/{\cal S}$ which has kernel ${\cal S}$ and which is universal for exact functors from ${\cal A}$ with kernel containing ${\cal S}$.  We give a specific computation of this at Example \ref{A21}.

\subsection{The abelian category associated to a module} \label{secassocab}\marginpar{secassocab}

We can do better, obtaining a unique (to equivalence) abelian category associated to a module $M$.  This is, in a very strong sense, an invariant of the module; compare with the ring associated to a module which, one may argue, is defined at most up to Morita equivalence.

The kernel of $\widetilde{M}$, ${\cal S}_M = \{ F\in {\rm Ab}({\cal R}): \widetilde{M}F=0\}$, is a Serre subcategory of ${\rm Ab}({\cal R})$ and there is a factorisation of $\widetilde{M}$ as a composition of exact functors through the quotient category ${\cal A}(M) = {\rm Ab}({\cal R})/{\cal S}_M$.

$\xymatrix{{\cal R} \ar[dd]_M \ar[r]^j & {\rm Ab}({\cal R}) \ar[ddl]^{\widetilde{M}} \ar[dr] \\
& & {\cal A}(M) \ar[dll]^{\widehat{M}} \\  {\bf Ab}}$

This category ${\cal A}(M)$ is the unique-to-equivalence skeletally small abelian category canonically associated to the module $M$.  I will give an interpretation in Sections \ref{secimag} and \ref{secmodth} of $\widetilde{M}$ and $\widehat{M}$ which explains the sense in which the modules seen in Section \ref{secKT}  are ``essentially the same" (in particular they do have the same associated abelian category ${\cal A}(M)$ up to equivalence).

That is how we can produce, from a ``module" in the usual sense, an exact functor on a small abelian category.  To go in the other direction, we start with a skeletally small abelian category ${\cal A}$ and an exact functor $G:{\cal A} \rightarrow {\bf Ab}$.  We use the fact (e.g.~\cite[2.18]{PreADC}) that every such abelian category ${\cal A}$ has the form (non-uniquely) ${\rm Ab}({\cal R})/{\cal S}$ for some skeletally small preadditive ${\cal R}$ and some Serre subcategory ${\cal S}$ of ${\rm Ab}({\cal R})$.  Since the localisation functor ${\rm Ab}({\cal R}) \rightarrow {\rm Ab}({\cal R})/{\cal S}$ is exact, composition with $G$ gives an exact functor on ${\rm Ab}({\cal R})$ and this exact functor restricts to an additive functor, $M$ say, on ${\cal R}$ - that is, a module in the usual sense.  Furthermore, to that additive functor $M$ on ${\cal R}$ we can apply the first process, obtaining $\widehat{M}:{\cal A}(M) \rightarrow {\bf Ab}$ and one may check that this is equivalent to the exact functor, $G'$ say, induced by $G$ on ${\cal A}/{\rm ker}(G)$.  That is, there is an equivalence of categories $\theta:{\cal A}(M) \rightarrow {\cal A}/{\rm ker}(G)$ such that $G'\theta$ is equivalent to $\widehat{M}$.

$
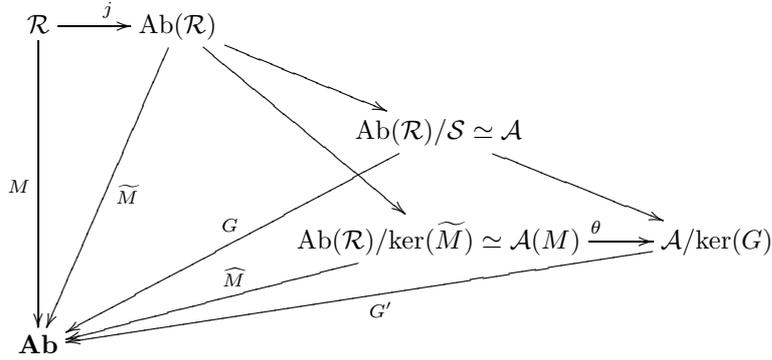
\xymatrix{{\cal R} \ar[ddd]_M \ar[r]^j & {\rm Ab}({\cal R}) \ar[dddl]^{\widetilde{M}} \ar[ddr] \ar[dr] \\
& & {\rm Ab}({\cal R})/{\cal S} \simeq {\cal A} \ar[ddll]_{G} \ar[dr] \\ & & {\rm Ab}({\cal R})/{\rm ker}(\widetilde{M}) \simeq {\cal A}(M) \ar[dll]_{\widehat{M}} \ar[r]^{\,\,\,\,\,\,\,\,\,\,\theta} &  {\cal A}/{\rm ker}(G) \ar[dlll]^{G'}\\{\bf Ab}}$

Thus we have the equivalence of the two definitions.

\vspace{4pt}

It is rather more interesting to try to compose these processes in the other direction since it raises the question: if we start with an additive functor $M: {\cal R} \rightarrow {\bf Ab}$, then can we recover $M$ and ${\cal R}$ from $\widehat{M}: {\cal A}(M) \rightarrow {\bf Ab}$?  The answer is that we cannot, not even up to Morita equivalence, the point being that, under our proposed definition, a module may be given in many equivalent but very different ways.  That is, given modules $M: {\cal R} \rightarrow {\bf Ab}$ and $M_1: {\cal R}_1 \rightarrow {\bf Ab}$, if there is an equivalence $H:{\cal A}(M) \rightarrow {\cal A}(M_1)$ such that $\widehat{M_1}H$ and $\widehat{M}$ are equivalent functors to ${\bf Ab}$, then we regard $M$ and $M_1$ as {\bf the same module}.

\subsection{The free abelian category as a functor category} \label{secfreeabfun}\marginpar{secfreeabfun}

The free abelian category on ${\cal R}$, constructed originally in \cite{FreydLJ}, has, following \cite{Adel}, a nice realisation as the category of finitely presented functors on finitely presented ${\cal R}$-modules: ${\rm Ab}({\cal R}) = ({\cal R}\mbox{-}{\rm mod}, {\bf Ab})^{\rm fp}$.  The embedding $j$ of ${\cal R}$ into $({\cal R}\mbox{-}{\rm mod}, {\bf Ab})^{\rm fp}$ is the composition of two Yoneda embeddings, first taking each object $p$ of ${\cal R}$ to the representable functor = projective left ${\cal R}$-module $(p,-)$ (but thought of as an object of $({\cal R}\mbox{-}{\rm mod})^{\rm op}$) and then taking that object to the corresponding representable functor.  The net result is $p \mapsto ((p,-),-)$ with an obvious definition on morphisms.  If, for example, ${\cal R}$ is just a ring $R$ with one object then this object $\ast$ is first sent to $(\ast,-)$, which is just the projective module $_RR$, then is sent on to the forgetful functor $(_RR,-)$ from $R\mbox{-}{\rm mod}$ to ${\bf Ab}$.

The extension process from ${\cal R}$-modules $M: {\cal R}\rightarrow {\bf Ab}$ to exact functors on ${\rm Ab}({\cal R}) = ({\cal R}\mbox{-}{\rm mod}, {\bf Ab})^{\rm fp}$ also is a two-step process which is described in detail elsewhere (e.g.~\cite[\S 4]{PreMAMS}), so I don't repeat it here.  I do, however, give details of one, very simple, example next.  After that, I describe another way of viewing the extension process, using a different incarnation of the free abelian category.

\begin{example} \label{A21}\marginpar{A21}
Take ${\cal R}$ to be the path category of the quiver $A_2$ (Example \ref{A20} but we now write the objects differently).  First consider the functor ${\cal R}^{\rm op} \rightarrow ({\cal R}\mbox{-}{\rm mod})$ (for convenience, I present it this way rather than ${\cal R} \rightarrow ({\cal R}\mbox{-}{\rm mod})^{\rm op}$) given by taking the object $i$ ($i=1,2$) to the representable functor/projective module $(i,-)=P_i$.  Explicitly these are given as follows, where we use $e_i$ for the identity map at $i$ and $\alpha$ for the arrow from $1$ to $2$.

$(1,-)$:  $(1,1)=Ke_1$, $(1,2)=K\alpha$, $(1,1) \xrightarrow{(1,\alpha)} (1,2): e_1 \mapsto \alpha$.

$(2,-)$:  $(2,1)=0$, $(2,2)=Ke_2$

$(\alpha,-): (2,-) \rightarrow (1,-)$:  has components $(\alpha,-)_1 =0 : (2,1)=0 \rightarrow (1,1)=Ke_1$ and $(\alpha,-)_2: (2,2)=Ke_2 \rightarrow (1,2)=K\alpha: e_2 \mapsto \alpha$.

We have the exact sequence $$0 \rightarrow (2,-) \xrightarrow{(\alpha,-)} (1,-) \rightarrow (1,-)/{\rm im}(\alpha,-) \rightarrow 0$$ in ${\cal R}\mbox{-}{\rm mod}$, where the cokernel may be identified as the simple left module $S_1$, given by $S_1(1) =Ke_1$, $S_1(2)=0$.  We write the sequence in more familiar form:  $0 \rightarrow P_2 \xrightarrow{i} P_1 \xrightarrow{\pi} S_1 \rightarrow 0$.

We know that $A_2$ has finite representation type and that this is the complete list of indecomposable modules, with every module being a direct sum of copies of these three modules.  In general, we would have to consider cokernels of arbitrary maps between direct sums of representable functors in order to find all the indecomposables.

That, the construction of ${\cal R}\mbox{-}{\rm mod}$, was the first stage of the construction of the free abelian category ${\rm Ab}({\cal R})$.  Now we just repeat the process, taking each finitely presented left ${\cal R}$-module $M$ to the representable functor $(M,-): {\cal R}\mbox{-}{\rm mod} \rightarrow {\bf Ab}$ and forming cokernels to obtain all the finitely presented functors from ${\cal R}\mbox{-}{\rm mod}$ to ${\bf Ab}$.

Consider, for instance, the exact sequence $0 \rightarrow S_1 \rightarrow P_1 \rightarrow P_2 \rightarrow 0$ in $({\cal R}\mbox{-}{\rm mod})^{\rm op}$.  This goes to the, non-exact, sequence $0 \rightarrow (S_1,-) \xrightarrow{(\pi,-)} (P_1,-) \xrightarrow{(i,-)} (P_2,-) \rightarrow 0$ in $({\cal R}\mbox{-}{\rm mod}, {\bf Ab})^{\rm fp}$.  We can obtain a couple of indecomposable, indeed simple, functors from this sequence.  First, the cokernel $(P_1,-)/{\rm im}(\pi,-)$ of $(\pi,-)$ is the functor given by $P_1 \mapsto K$, $P_2, S_1 \mapsto 0$; second the cokernel of $(i,-)$, which is the functor given on objects by $P_2 \mapsto K$, $P_1, S_1 \mapsto 0$.

In fact, that's it - there are no more that these five indecomposable finitely presented functors from ${\cal R}\mbox{-}{\rm mod}$ to ${\bf Ab}$.  So, in this particular case, we can completely decribe the free abelian category.  We do that in a less bare-hands way, and give a clearer picture of it, at Example \ref{A2lang}.
\end{example}

\subsection{The free abelian category {\it via} systems of linear equations} \label{secfreeabpp}\marginpar{secfreeabpp}

There is a very different-seeming construction of the free abelian category, arising from the model theory of modules, see \cite{KP1}, \cite{HerzDual}, \cite{BurThes} (one can also take a categorical logic approach, see \cite[\S 2.2]{BVCL} in particular).  In this construction the objects of the free abelian category are pairs of projected systems of linear equations; I explain this.

\subsubsection{pp formulas}

Let $R$ be a ring [I will parenthetically say how to modify everything for general preadditive ${\cal R}$].  A {\bf (homogeneous) system of} $R$-{\bf linear equations} is an expression of the form $$\bigwedge_{j=1}^m \, \sum_{i=1}^n \, r_{ij}x_i=0$$  where $x_1, \dots, x_n$ are variables (which can run over the elements of any particular $R$-module) and the $r_{ij}$ are the elements (or corresponding multiplication maps) of $R$.  (The symbol $\bigwedge$ is read ``and", as is $\wedge$, where the use of $\bigwedge/\wedge$ is like that of $\sum/+$.)

[If the preadditive category ${\cal R}$ has more than one object, then each variable has a {\bf sort} - the sorts are labelled by the objects $p$ of ${\cal R}$ and a variable of sort $p$ runs over elements of the group $M(p)$ in any particular ${\cal R}$-module $M$.  So then, for the above expression to make sense, each morphism $r_{ij}$ of ${\cal R}$ must be an arrow from $p$, the sort of $x_i$, to some $q=q(j)$ - so ``0" in the $j$th equation means the zero element in $M(q)$.  For instance, over $A_2$, $(x_1-\lambda x_2=0) \wedge (\alpha x_1 -x_3=0)$, with $\lambda$ some element of $K$, is a system of linear equations, where $x_1$ and $x_2$ both have sort labelled by the object 1 of ${\cal R}$ and $x_3$ has sort labelled by 2.]

The solution set in any module $M$ of such a system of linear equations is a subgroup of $M^n$ (not necessarily a submodule unless $R$ is commutative).  The projection of such a solution set to, say, the first $k\leq n$ coordinates is what is called a {\bf pp-definable subgroup} of $M$ (strictly, a subgroup of $M^k$ pp-definable in $M$).  The terminology ``pp", short for ``positive primitive", refers to the form of the condition $$\exists x_{k+1}, \dots , x_n \, \bigwedge_{j=1}^m \, \sum_{i=1}^n \, r_{ij}x_i=0$$ which defines this set.  If we think of this as a condition on the variables $x_1, \dots, x_k$ then we refer to the condition itself as a {\bf pp condition} or {\bf pp formula}, using notation such as $\phi$ or, if we want to show the free (i.e.~unquantified) variables, $\phi(x_1, \dots, x_k)$. The solution set in $M$ - the corresponding pp-definable subgroup - that is $$\{ (a_1, \dots, a_k)\in M^k: \exists a_{k+1}, \dots, a_n \in M \,\, \sum_{i=1}^n \, r_{ij}a_i=0 \mbox{ for each } j=1, \dots, m\}$$ we write as $\phi(M)$.

[In the many-sorted case, where ${\cal R}$ has more than one object, the solution set of such a pp formula in an ${\cal R}$-module $M$ will be a subset of a finite product of abelian groups each of the form $M(p)$ for $p$ an object of ${\cal R}$.  In the $A_2$ example above, the solution set of the system of two equations is a subgroup of $M(1)\oplus M(1) \oplus M(2)$; if we quantify out the variables $x_1$ and $x_2$, to get the pp formula $\exists x_1, x_2 \, \big((x_1-\lambda x_2=0) \wedge (\alpha x_1 -x_3=0)\big)$, then the solution set of this, in any module $M$, is a subgroup of $M(2)$, namely the image of $\alpha$.]

If $\psi(\overline{x})$ and $\phi(\overline{x})$ are pp formulas with the same free variables, we write $\psi \leq \phi$ if, for every module $M$, $\psi(M)$ is contained in $\phi(M)$, that is, if $\psi$ is a stronger condition than $\phi$.  There is a characterisation (see \cite[1.1.13]{PreNBK}) of this in terms of the coefficients $r_{ij}$ of the respective systems of equations.  If $\psi \leq \phi$ we refer to this as a {\bf pp-pair}, writing $\phi/\psi$.

\subsubsection{The category of pp-pairs}

We associate to any module category ${\cal R}\mbox{-}{\rm Mod}$ the {\bf category}, ${\mathbb L}({\cal R}\mbox{-}{\rm Mod}) = {\mathbb L}_{{\cal R}^{\rm op}}^{\rm eq+}$, {\bf of pp-pairs}.  The objects of this category are the pp-pairs $\phi/\psi$.  The morphisms are the pp-defined maps; literally they are certain equivalence classes of pp formulas.  Precisely, a morphism from $\phi(\overline{x})/\psi(\overline{x})$ to $\phi'(\overline{y})/\psi'(\overline{y})$ is given by a pp formula $\rho(\overline{x},\overline{y})$ which defines a (necessarily additive) map from $\phi(M)/\psi(M)$ to $\phi'(M)/\psi'(M)$ for every module $M$.  The conditions for a pp formula to define a map (rather than just a relation) are simply expressed:  $ \phi(\overline{x}) \leq \exists \overline{y} \,\, \rho(\overline{x}, \overline{y})$, $ \big(\rho (\overline{x},\overline{y}) \wedge \phi (\overline{x})\big) \leq \phi '(\overline{y})$ and $ \big( \rho (\overline{x},\overline{y}) \wedge \psi (\overline{x})\big) \leq \psi '(\overline{y}) $  (in every module).  Since different pp formulas can define the same function in every module, we take such formulas up to this equivalence as the arrows of ${\mathbb L}_{{\cal R}^{\rm op}}^{\rm eq+}$.  All these conditions, though initially they might seem not checkable since they refer to every module, are equivalent to certain ${\cal R}$-bilinear constraints on the coefficients appearing in the formulas involved (use \cite[1.1.13]{PreNBK}).

This category is, in fact \cite{HerzDual}, abelian and is equivalent to the category $({\cal R}\mbox{-}{\rm mod}, {\bf Ab})^{\rm fp}$ (\cite[3.2.5]{BurThes}; see, e.g., \cite[10.2.30]{PreNBK}), hence is an incarnation of ${\rm Ab}({\cal R})$.  The equivalence with the functor category arises as follows.

One easily checks that homomorphisms preserve solution sets of pp formulas so, for any pp formula $\phi$, the assignment $M\mapsto \phi(M)$ defines a functor, which we denote $ F_\phi $, from ${\cal R}\mbox{-}{\rm Mod}$ to $ {\bf Ab}$.  Because (taking the solution set of) a pp formula commutes with direct limits and since every module is a direct limit of finitely presented modules, this functor is determined by its restriction to the subcategory, $ {\cal R}\mbox{-}{\rm mod}$, of finitely presented modules. We will use the same notation, $F_\phi$, for this restriction.  This is a subfunctor of a power of the forgetful functor [in the case that ${\cal R}$ has one object; in the general case, a subfunctor of some product of representable-by-objects-of-${\cal R}$ functors].  One may check (using ``free realisations" see \cite[\S 1.2.2]{PreNBK}) that this functor is finitely generated so, since the functor category $({\cal R}\mbox{-}{\rm mod}, {\bf Ab})^{\rm fp}$ is locally coherent \cite{AusCoh}, $F_\phi$ is finitely presented.   Furthermore, if $ \psi \leq \phi $ then $ F_\psi$ is a subfunctor of $F_\phi  $ and we can form the quotient functor $ F_{\phi /\psi }=F_\phi /F_\psi $, which takes a module $M$ to $\phi(M)/\psi(M)$. As a quotient of finitely presented functors, this is finitely presented.  It may be shown that every finitely presented functor has this form and that, moreover, the natural transformations between finitely presented functors are given by pp formulas (see \cite[10.2.30]{PreNBK}).  This is part of an extensive equivalence between the functorial and model-theoretic views of modules.

\section{The context in which a module sits - definable categories, and the functors between these}

\subsection{Definable categories} \label{secdefcat}\marginpar{secdefcat}

When we define a module in the usual way, thereby fixing the ring ${\cal R}$, it is natural to regard the module as sitting inside the category ${\cal R}\mbox{-}{\rm Mod}$ of all additive functors from ${\cal R}$ to ${\bf Ab}$.  This is a typical Grothendieck abelian category with a generating set of finitely generated projectives.  On the other hand, the definition of a module $M$ as an exact functor on an abelian category ${\cal A}$ naturally places $M$ within the category, ${\rm Ex}({\cal A}, {\bf Ab})$, of all such functors.  The categories which arise this way are the ({\bf exactly}) {\bf definable} ({\bf additive}) {\bf categories}.  These include module categories and all finitely accessible additive categories (see \cite{AdRo}) with products but are more general than that - for example, they may have no finitely presented objects apart from $0$, see Example \ref{inj}.  The fact that there is a canonical choice for ${\cal A}$, given $M$, gives a minimal definable category ${\rm Ex}({\cal A}(M), {\bf Ab})$ in which $M$ lives; we will see below that there are various senses in which this category is generated by $M$.

The fact that definable categories are accessible follows from the Downwards L\"{o}wenheim-Skolem Theorem (see any text on model theory, e.g.~\cite{Hod}).  But accessible categories with products need not be definable; for example a Grothendieck category is definable iff it is finitely accessible \cite[3.6]{PreADC}.  As further examples of (finitely accessible) definable categories, we have the category of comodules over a $K$-coalgebra and, over many schemes, the category of quasicoherent sheaves.

Definable categories first arose as certain subcategories of module categories:  we say that a full subcategory (closed under isomorphism) of a module category ${\cal R}\mbox{-}{\rm Mod}$ is a {\bf definable subcategory} if it satisfies the equivalent conditions of the next result, where an embedding $f:M \rightarrow N$ between left ${\cal R}$-modules is {\bf pure} if for every (finitely presented) right ${\cal R}$-module $L$, the morphism $L\otimes_{\cal R} M \xrightarrow{1_L\otimes f} L\otimes_{\cal R} N$ is an embedding (of abelian groups); $\otimes_{\cal R}$ is illustrated in the example that follows.

\begin{theorem}\label{chardefsub} \marginpar{chardefsub} (see \cite[3.4.7]{PreNBK}) The following are equivalent for a subcategory ${\cal D}$ of ${\cal R}\mbox{-}{\rm Mod}$:

\noindent (i) ${\cal D}$ is closed in ${\cal R}\mbox{-}{\rm Mod}$ under direct products, direct limits and pure submodules;

\noindent (ii) ${\cal D}$ is closed in ${\cal R}\mbox{-}{\rm Mod}$ under ultraproducts, finite direct sums and pure submodules;

\noindent (iii) there is a set $\Phi$ of functors in $({\cal R}\mbox{-}{\rm mod}, {\bf Ab})^{\rm fp}$ such that ${\cal D} = \{ M\in {\cal R}\mbox{-}{\rm Mod} : FM=0 \mbox{ for all } F\in \Phi\}$;

\noindent (iii) there is a set $\Phi$ of pp-pairs (in the language of ${\cal R}$-modules) such that ${\cal D} = \{ M\in {\cal R}\mbox{-}{\rm Mod} : \phi(M)/\psi(M)=0 \mbox{ for all } \phi/\psi \in \Phi\}$.
\end{theorem}

\begin{example}\label{A2tens}\marginpar{A2tens}  I recall how to define tensor product over a small preadditive category, using the path category, ${\cal R}$, of $A_3$ to illustrate this.

A left ${\cal R}$-module $M$, that is a functor from ${\cal R}$ to $K\mbox{-}{\rm Mod}$, can usefully be identified with its ``image", meaning the corresponding representation $M1 \xrightarrow{M\alpha} M2 \xrightarrow{M\beta} M3$ of $A_3$.  A right module $L$ can be identified with a representation, $L1 \xleftarrow{L\alpha} L2 \xleftarrow{L\beta} L3$, of the opposite quiver.

The first part of the definition of $(-)\otimes_{\cal R}M$ is that its value on a representable right module $(-,i)$ is $Mi$; for instance if $i=1$ then $(-,1)$ is being identified with the representation $K \leftarrow 0 \leftarrow 0$ so $(-,1)\otimes M  = M1$.  Taking $i=2$, hence the representation $K \xleftarrow{1} K \leftarrow 0$, we have $ (-,2)\otimes M= M2$ and, similarly, $ (-,3)\otimes M=M3$.

The rest of the definition of $-\otimes M$ is that it is right exact, which is enough because we've just said how it is defined on the indecomposable projectives.  For instance, if $L$ is $0 \leftarrow K \xleftarrow{1} K$, so has projective presentation $(-,1) \xrightarrow{(-,\alpha\beta)} (-,3) \rightarrow L \rightarrow 0$, then, tensoring with $M$, we obtain the exact sequence $M1 \xrightarrow{M(\alpha\beta)} M3 \rightarrow ML \rightarrow 0$.  So, if $M$ is the representation $K \xrightarrow{1} K \xrightarrow {1} K$ then $L\otimes M = 0$ and, if $M$ is any other indecomposable representation, then $L\otimes M = M3$.
\end{example}

\begin{theorem}\label{defequiv} \marginpar{defequiv} (see \cite[18.1.6]{PreNBK}) A category is exactly definable iff it is equivalent to a definable subcategory of some module category ${\cal R}\mbox{-}{\rm Mod}$.  We say simply that such a category is {\bf definable}.
\end{theorem}

If $M \in {\cal R}\mbox{-}{\rm Mod}$ then write $\langle M \rangle$ for the definable subcategory of ${\cal R}\mbox{-}{\rm Mod}$ {\bf generated by} $M$ (in the sense of being the smallest definable subcategory which contains $M$).  Every definable category has a generator in this sense, indeed an {\bf elementary cogenerator}, meaning an object $N$ such that every $M\in {\cal D}$ purely embeds in a direct product of copies of $N$.  We say that an object $N$ is {\bf pure-injective} if it is injective over pure embeddings.

\begin{theorem} (see \cite[\S 3.4,5.3.52]{PreNBK}) Every definable category has the form $\langle M \rangle$ for some module $M$.  Since every definable category is closed under pure-injective hulls, we may take $M$ to be pure-injective.  Indeed every definable category has an elementary cogenerator, which may be taken to be a direct product of indecomposable pure-injectives, alternatively the pure-injective hull of a direct sum of indecomposable pure-injectives.
\end{theorem}

\begin{example} Within the category ${\mathbb Z}\mbox{-}{\rm Mod} = {\bf Ab}$ of abelian groups, examples of definable subcategories are:

\noindent $\langle {\mathbb Z}\rangle$, the torsionfree abelian groups;

\noindent $\langle {\mathbb Q}/{\mathbb Z}\rangle$, the category of divisible abelian groups;

\noindent $\langle {\mathbb Q}\rangle$, the intersection of these;

\noindent $\langle \overline{{\mathbb Z}_{(p)}}\rangle =  \langle {\mathbb Z}_{(p)}\rangle$, the localisation of ${\bf Ab}$ at any nonzero prime $p$;

\noindent $\langle A\rangle$, the category of direct sums of copies of direct summands of $A$, where $A$ is any finite abelian group.

There are further definable subcategories of ${\bf Ab}$ and these are, as are the definable subcategories of any definable category, in natural bijection with the closed sets of the associated Ziegler spectrum.  That is a topological space with points being the isomorphism classes of indecomposable pure-injectives and with topology defined using pp-pairs, as originally introduced by Ziegler \cite{Zie}, but which can be defined in a variety of ways, see \cite[\S\S 5.1, 12.4]{PreNBK}.

The category of torsion abelian groups is not a definable subcategory of ${\bf Ab}$ though, being a finitely accessible category with products (the finite groups are its finitely presented $\varinjlim$-generators) it is a definable category (which is not `definably embedded' in ${\bf Ab}$).
\end{example}

\subsection{The intrinsic structure of a definable category} \label{defint}\marginpar{defint}

If ${\cal D}$ is a definable subcategory of a module category ${\cal R}\mbox{-}{\rm Mod}$ then the usual pure-exact structure on the latter restricts to an exact structure on ${\cal D}$ which is, in fact, intrinsic to ${\cal D}$.  That follows directly from the fact that a sequence $0 \rightarrow A \rightarrow B \rightarrow C \rightarrow 0$ in ${\cal D}$ is pure-exact in the sense of ${\cal R}\mbox{-}{\rm Mod}$ iff some ultraproduct of it is split.  Since ultraproducts are direct limits of products, and ${\cal D}$ is closed in ${\cal R}\mbox{-}{\rm Mod}$ under these operations, this definition of purity may be used without any reference to such a containing category ${\cal R}\mbox{-}{\rm Mod}$.  Thus there is an intrinsic theory of purity in any definable category.

Such a category ${\cal D}$ also has an intrinsic model theory, with language based on the category ${\cal A}(M)$ where $M$ is any generator for ${\cal D}$ as a definable category.  This category does not depend on choice of generator $M$, being equivalent to the category $({\cal D}, {\bf Ab})^{\prod \rightarrow}$ of additive functors from ${\cal D}$ to ${\bf Ab}$ which commute with direct products and direct limits (\cite[12.10]{PreMAMS}).  We write ${\rm fun}({\cal D})$ for this skeletally small abelian category.  So, in the case of a module category, we have the following ways of obtaining the free abelian category of a ring.

\begin{theorem}\label{equiv3}\marginpar{equiv3} For any skeletally small preadditive category ${\mathcal R}$ the categories $ ({\cal R}\mbox{-}{\rm mod},{\bf Ab})^{\rm fp}$, ${\mathbb L}^{\rm eq+}_{{\mathcal R}^{\rm op}}$, ${\rm Ab}({\mathcal R})$ and ${\rm fun}\mbox{-}{\mathcal R}^{\rm op} = ({\cal R}\mbox{-}{\rm Mod}, {\bf Ab})^{\prod \rightarrow}$ are equivalent.
\end{theorem}

The canonical (multi-sorted) language for a definable category ${\cal D}$ has one sort for each object of ${\rm fun}({\cal D})$, an abelian group structure on each sort, and a function symbol for each arrow of ${\rm fun}({\cal D})$.  This general kind of language is discussed in Section \ref{secmodth}; we consider a specific example here.

\begin{example}\label{A2lang}\marginpar{A2lang}  We recompute the functor category ${\rm fun}({\cal D})$ of the category of modules ${\cal D} = KA_2\mbox{-}{\rm Mod}$ over the path category of $A_2$, using some more general theory than we did with the direct computations at Example \ref{A21}.

We use the fact that, for a ring $R$ of finite representation type, such as $KA_2$, the functor category is equivalent (see, e.g., \cite[4.9.4]{BenBk1}, \cite[\S VI.5]{ARS}) to the category of right modules over the {\bf Auslander algebra} of $ R $; this is the endomorphism ring $  S={\rm End}(M) $ where $ M $ is a direct sum of one copy of each of the indecomposable (finitely generated) $R$-modules.  In our example, and continuing the notation of Example \ref{A21}, the Auslander algebra is $ S={\rm End}(P_1\oplus P_2 \oplus S_1)$.  Then, by direct computation (let $S$ act on the right as a $3\times 3$ matrix ring and decompose it into indecomposable projectives = rows), $S$ is the path algebra of the quiver $\bullet_1 \xleftarrow{\pi} \bullet_2 \xleftarrow{i} \bullet_3$ with relation $i\pi =0$.  The Auslander-Reiten quiver of ${\rm Mod}\mbox{-}S$ is the following, where $T_i, Q_i, J_i$ are respectively the simple module at vertex $i$, the projective cover of $T_i$, the injective hull of $T_i$.

$
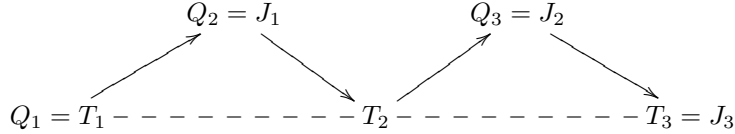
\xymatrix{ &  Q_2=J_1 \ar[dr] & & Q_3=J_2 \ar[dr] \\  Q_1=T_1 \ar[ur] \ar@{--}[rr]
 & & T_2 \ar[ur] \ar@{--}[rr] & & T_3=J_3}$

One can check\footnote{For example, first note that $Q_2$ is the functor $(P_1,-)$, hence has action as given, and similarly for $Q_3$.  Then, hom the exact sequence $0\rightarrow P_2 \xrightarrow{i} P_1 \xrightarrow{\pi} S_1 \rightarrow 0$ into $M$ to get that $Q_1(M) = (S_1,M) = {\rm ker}(i,M)$ which, since $i$ is just multiplication by $\alpha$, gives the pp-description of this functor.  Since $T_2$ is the cokernel of the embedding of $Q_1$ into $Q_2$, it can be described as the quotient $M\mapsto e_1M/\{ m\in e_1M: \alpha m=0\}$, which is isomorphic (definably, by multiplication by $\alpha$) to $\alpha M$.  And the description of $T_3$ is also immediate.} that these are, as functors and as pp-pairs, as follows (where $M$ can be the module above or, since the action on $M$ determines that on any module, arbitrary):

\noindent $Q_1=(S_1,-)= (\alpha x=0)/(x=0): M\mapsto \{ m\in e_1M: \alpha m=0\}$;

\noindent $Q_2 =(P_1,-) = (x=e_1x)/(x=0): M \mapsto e_1M$;

\noindent $T_2 =(P_1,-)/{\rm im}(\pi,-)\simeq (P_1,-)/(S_1,-) \simeq (\alpha|x)/(x=0): M \mapsto \alpha M$;

\noindent $Q_3=(P_2,-) =(x=e_2x)/(x=0):M\mapsto e_2M$;

\noindent $T_3= (P_2,-)/{\rm im}(i,-)\simeq (P_2,-)/(P_1,-) = (x=e_2x)/(\alpha|x): M\mapsto e_2M/\alpha M$.

In particular $S$ is of finite representation type and so there are just these 5 indecomposable pp-pairs, with every pp sort, that is object of ${\mathbb L}_{(KA_2)^{\rm op}}^{\rm eq+}$ being a direct sum of copies of these.  In general, even if $R$ is of finite representation type, the Auslander algebra of $R$ need not be, so this is a very pleasant situation in which the category of pp-sorts can be described completely.  A similar computation, for the algebra $K[\epsilon: \epsilon^2=0]$, can be found in \cite[\S 6.8]{Perera}.
\end{example}

\subsection{Localisation} \label{secloc}\marginpar{secloc}

Localisation/relativisation is relevant when we move from a definable category to a definable subcategory.  The restriction/relativisation of functors (that is, pp-pairs) to the subcategory exactly corresponds to Serre localisation of the associated functor categories.

Precisely, suppose that $ {\mathcal D}$ is a definable subcategory of, say, ${\cal R}\mbox{-}{\rm Mod}$.  By $ {\mathbb L}^{\rm eq+}({\mathcal D}) $ we denote the category of pp-pairs for $ {\mathcal D} $.  The objects of $ {\mathbb L}^{\rm eq+}({\mathcal D}) $ may be taken to be the same as those of ${\mathbb L}_{{\cal R}^{\rm op}}^{\rm eq+}$ (alternatively, the pp-pairs obtained from the intrinsic model theory of ${\mathcal D}$).  The morphisms from sort $ \phi /\psi  $ to $ \phi '/\psi ' $ are the equivalence classes, meaning equivalence {\em on modules in} $ {\mathcal D}$, of pp formulas which define a relation from the group defined by $ \phi /\psi  $ to that defined by $ \phi '/\psi ' $ which is total and functional in every $ D\in {\mathcal D}$.

\begin{example}  Take ${\cal D}$ to be the definable subcategory ${\mathbb Z}_{(2)}\mbox{-}{\rm Mod}$ of ${\mathbb Z}\mbox{-}{\rm Mod}$.  Since every prime other than 2 is invertible when acting on the objects of ${\cal D}$, the sort $(y=y)/(3|y)$ must be isomorphic to the zero object of $ {\mathbb L}^{\rm eq+}({\mathcal D}) $.  Represent the zero object as, say, $(x=x)/(x=x)$.  The pp formula $\rho(x,y)$ which is just $y=x$, regarded as defining a relation from $(x=x)/(x=x)$ to $(y=y)/(3|y)$ is certainly not functional on ${\mathbb Z}\mbox{-}{\rm Mod}$ - if $M$ is a ${\mathbb Z}$-module then $\rho$ relates the only element of $M/M$ to every element of $M/3M$ - but it is functional on modules in ${\cal D}$ since, for these, $M/3M=0$.  For an example of an arrow of $ {\mathbb L}^{\rm eq+}({\mathcal D}) $ which is not in $ {\mathbb L}^{\rm eq+}({\mathbb Z}\mbox{-}{\rm Mod})$, take multiplication by the inverse of $3$ on any sort, say on $(x=x)/(x=0)$ where the relation defined by the formula $\sigma(x,y)$, being $3y=x$, gives a new arrow.
\end{example}

\begin{theorem}\label{ppsortquotpp}\marginpar{ppsortquotpp} (see \cite[12.3.20]{PreNBK}) If $ {\mathcal D} $ is a definable subcategory of $ {\mathcal R}\mbox{-}{\rm Mod} $ and $ {\mathcal S}_{{\mathcal D}} $ denotes the full subcategory of $ {\mathbb L}^{\rm eq+}_{{\mathcal R}^{\rm op}}$ on those pp-pairs which are {\bf closed on} (i.e.~zero when evaluated on) ${\mathcal D}$ then the quotient category $  {\mathbb L}^{\rm eq+}_{{\mathcal R}^{\rm op}}/{\mathcal S}_{{\mathcal D}} $ is naturally equivalent to $ {\mathbb L}^{\rm eq+}({\mathcal D})$.
\end{theorem}

The naturality of the equivalence is with respect to the actions on $ {\mathcal D}$: if $ \phi/\psi $ is a pp-pair for ${\mathcal R}$-modules then the action of its image in the quotient category is given by the same pp-pair but now restricted to ${\mathcal D}$.

Every small abelian category arises thus, as the category of pp-pairs for some definable category, see \cite[2.18]{PreADC}.

We give an example of computing such a localisation.

\begin{example} \label{A2langloc}\marginpar{A2langloc}
Take ${\cal D}$ to be the definable subcategory generated by the indecomposable representation $K\xrightarrow{1}K$ of $KA_2$.  The kernel of the localisation from ${\rm fun}(KA_2\mbox{-}{\rm Mod}) ={\mathbb L}_{(KA_2)^{\rm op}}^{\rm eq+}$ to ${\rm fun}(\langle K\xrightarrow{1}K\rangle)$ is the Serre subcategory generated by those indecomposables which are $0$ on this representation, that is (continuing the notation from Example \ref{A2lang}), by $Q_1$ and $T_3$.  Therefore, in the quotient category, the epimorphism from $Q_2$ to $T_2$ becomes an isomorphism (it had kernel $Q_1$ which is now zero); also the embedding of $T_2$ in $Q_3$ becomes an isomorphism (it had cokernel $T_3$).  So we see that there is just one indecomposable object of ${\rm fun}( \langle K\xrightarrow{1}K\rangle)$, as we should expect since, note, $\langle K\xrightarrow{1}K\rangle$ is equivalent to the category $K\mbox{-}{\rm Mod}$ of $K$-vector spaces and that, one may check (or see \cite[10.2.38]{PreNBK}), has $K\mbox{-}{\rm mod}$ for its functor category.
\end{example}

\subsection{Imaginaries} \label{secimag}\marginpar{secimag}

If $M$ is an ${\cal R}$-module, we may view the corresponding exact functor $\widehat{M}: {\cal A}(M) \rightarrow {\bf Ab}$ as an enriched version of $M$.  It has been enriched by the addition of many new sorts, namely the various $\phi(M)/\psi(M)$ for $\phi/\psi \in {\mathbb L}^{\rm eq+}_{{\mathcal R}^{\rm op}}$ (or in the Serre quotient ${\cal A}(M)$ of that category), plus the pp-definable functions between these sorts.  The model-theoretic view is that these are new kinds of elements of $M$, ``imaginaries", see Section \ref{secmodth}.  They ``are essentially" elements of $M$ in the same sense that one might say this of $n$-tuples of elements of $M$.  This has proved to be an extremely useful point of view in model theory in general as well as in particular contexts.  In the context of modules/additive structures it fits particularly well with the functor-category approach to representation theory.

For example, from this point of view, if $A$ is a finitely presented module, then the elements of $(A,M)$ are essentially elements of $M$.  Also, if $A$ is ${\rm FP}_2$, then the elements of ${\rm Ext}^1(A,M)$ are essentially elements of $M$ (see \ref{tilting}).  From this perspective, in the example in Section \ref{secKT}, the elements of the $K\widetilde{A_1}$-representation build from a $K[T]$-module $M$ are essentially elements of $M$, that $K\widetilde{A_1}$-representation being an imaginary sort of $M$.  Indeed, its underlying group has the form $\phi(M)/\psi(M)$ for some pp-pair $\phi/\psi$ in the language of $K[T]$-modules, and the action of each element of $K\widetilde{A_1}$ on that group is defined by a pp formula.

Put another way, denoting by $M^{\rm eq+}$ the multi-sorted structure built in this way from $M$, this is just a slightly different way of viewing the functor $\widehat{M}: {\cal A}(M) \rightarrow {\bf Ab}$, which is evaluation of sorts at $M$.  Namely, it is the image of that functor.

\subsection{Interpretation functors} \label{secinterp}\marginpar{secinterp}

A functor $F: {\cal C} \rightarrow {\cal D}$ between definable categories is an {\bf interpretation functor} if it commutes with direct products and direct limits (it then follows that $F$ preserves pure-exact sequences and pure-injectivity).  So these are the structure-preserving functors between definable categories.  They are also the functors which are model-theoretic interpretations given by pp formulas - see Section \ref{secmodth}.  They precisely correspond, contravariantly, to the exact functors between the corresponding abelian functor categories - see Section \ref{sec2cats}.  The functors we saw in Section \ref{secKT} are easily seen to be interpretation functors.

\begin{example}  Given a morphism $\theta: R\rightarrow S$ of rings, this induces a restriction-of-scalars functor from ${\rm Mod}\mbox{-}S$ to ${\rm Mod}\mbox{-}R$, and that is clearly an interpretation functor.
\end{example}

\begin{example} \label{repemb}\marginpar{repemb}
If $_SB_R$ is an $(S,R)$-bimodule which, as a left $S$-module is finitely presented, then the functor $-\otimes_SB_R: S\mbox{-}{\rm Mod} \rightarrow R\mbox{-}{\rm Mod}$ is an interpretation functor:  since tensor is a left adjoint, it preserves direct limits, and tensoring with a finitely presented module preserves products \cite{Len}.

In particular, if $R$ and $S$ are finite-dimensional algebras then we say that $F=( _SB_R\otimes _S-): S\mbox{-}{\rm Mod} \rightarrow  R\mbox{-}{\rm Mod}$ is a {\bf representation embedding} if $ _SB $ is finitely generated projective and the restriction of $F$ to $S\mbox{-}{\rm mod}$ preserves indecomposability and reflects isomorphism (but we are not requiring that it must be full on isomorphisms).  (The concept is not restricted to finite-dimensional algebras.  For instance, if $S$ is $K[T]$ or $K\langle X,Y\rangle$ then the condition is usually required on finite-dimensional, rather than finitely presented, $S$-modules.)

What is not clear is whether a representation embedding $F$ is an equivalence, in the category of definable categories and interpretation functors, between its domain and its image.  That is, can the original category $S\mbox{-}{\rm Mod}$ be ``recovered definably" from its image?  It is an open question whether this is always the case, though it is true if we have a strict representation embedding \cite{PreWild} or, more generally, if $F$ is a finitely-controlled representation embedding \cite[4.3]{GrePre}.  From this one can deduce that the conjecture ``wild representation type implies undecidable theory of modules" is true for finitely-controlled-wild (including strictly wild) algebras.
\end{example}

I explain, in Section \ref{secmodth}, the sense in which interpretation functors are indeed ``interpretations".

\section{When is a module over a ring?} \label{secmodrng}\marginpar{secmodrng}

We say that a module $M$ is {\bf over a ring} if $M$ can be represented as a functor from a single-object preadditive category to ${\bf Ab}$, equivalently if there is a ring $R$ with 1 such that ${\cal A}(M)$ is a quotient of ${\rm Ab}(R)$ by a Serre subcategory.

\begin{example}\label{Ainfex0}\marginpar{Ainfex0} Let ${\mathcal R}$ be the $K$-linear path category of the quiver $A_\infty^\infty$:  \\ $\dots \rightarrow -1 \rightarrow 0 \rightarrow 1 \rightarrow 2 \rightarrow \dots$.  So, $(i,j) =K$ if $i\leq j$ and $=0$ otherwise.  The category of ${\mathcal R}$-modules is not equivalent to the category of modules over any ring with $1$ so there will be many ${\cal R}$-modules which are not over any ring.\footnote{This can be deduced, for example, from the fact that the Ziegler spectrum of this category is not compact:  the open sets determined by the pp-pairs $(e_ix=x)/(x=0)$ give a cover with no finite subcover but, if $R$ is a ring with 1, then its Ziegler spectrum is compact and so, therefore, is  the Ziegler spectrum of any definable subcategory, since it is a closed subset.}
\end{example}

We will say that an object $A$ in an abelian category ${\cal A}$ is an {\bf abelian generator} of ${\cal A}$ if ${\cal A}$ is the smallest, possibly non-full, abelian subcategory of ${\cal A}$ containing $A$ and all its endomorphisms.

\begin{prop}\label{modoverrng}\marginpar{modoverrng} Let $M$ be a module; then $M$ is over a ring iff there is an object $U$ of ${\cal A}(M)$ which is an abelian generator of ${\cal A}(M)$.
\end{prop}
\begin{proof}  First we recall why $(_RR,-)$ is an abelian generator of ${\rm Ab}(R)$.  Let $A\in R\mbox{-}{\rm mod}$ and choose a projective presentation $R^m \xrightarrow{h} R^n \xrightarrow{p} A \rightarrow 0$.  This induces an exact sequence $0 \rightarrow (A,-) \xrightarrow{(p,-)} (R,-)^n \xrightarrow{(h,-)} (R,-)^m$ so, as the kernel of $(h,-)$, the functor $(A,-)$ and its inclusion into $(R,-)^n$ is in the abelian subcategory generated by $(R,-)$.  Any morphism $(B,-) \rightarrow (A,-)$ has the form $(f,-)$ for some $f:A\rightarrow B$ in $R\mbox{-}{\rm mod}$.  Taking also a projective presentation of $B$, we can see, using that $(R,-)$ is injective (as well as projective), that $(f,-)$ is the restriction of some morphism between powers of $(R,-)$.
$\xymatrix{
0 \ar[r] & (A,-) \ar[r]^{(p,-)} & (R,-)^n  \ar[r]^{(h,-)} & (R,-)^m \\
0 \ar[r] & (B,-) \ar[r]^{(q,-)} \ar[u]^{(f,-)} & (R,-)^k  \ar[r] \ar[u]^{(f',-)} & (R,-)^l
}$

Consider the epi-mono factorisation of $(f,-)$; we claim that each map is in the abelian subcategory generated by $(R,-)$.  The epimorphism is there because it is the cokernel of the kernel of $(p,-)(f,-) = (f',-)(q,-)$ and the latter is a morphism in the abelian category generated by $(R,-)$.  It then follows from the cokernel property that the monomorphism also is in the subcategory.  So the composition, $(f,-)$ of these maps is in the subcategory, as required.  Since every object of ${\rm Ab}(R)$ is the cokernel of a map such as $(f,-)$, we have all the objects of ${\rm Ab}(R)$ in the abelian subcategory generated by $(R,-)$.  Finally, given a morphism $F_g \rightarrow F_f$ in ${\rm Ab}(R)$ we can lift it to a projective presentation and deduce from what we have already, that this morphism also is in the subcategory generated by ${\rm Ab}(R)$.

[The argument is equally valid in the case where general ${\cal R}$ replaces $R$, the conclusion being that ${\rm Ab}({\cal R})$ is the smallest abelian subcategory containing all the (projective-injective) functors $(P,-)$ where $P$ ranges over any generating set of the category ${\cal R}\mbox{-}{\rm proj}$ of finitely generated projective modules.)]

If ${\cal A}(M)$ is a Serre quotient of a category of the form ${\rm Ab}(R)$ then, since the forgetful functor $(_RR,-)$ is an abelian generator of ${\rm Ab}(R)$, the same is true of its image in ${\cal A}(M)$ (since the inverse image under an exact functor of an abelian subcategory is abelian).

For the converse, suppose that ${\cal A}(M)$ has an abelian generator $F$; set $R={\rm End}(F)$.  Let $M_F$ denote the functor from the preadditive category $R$ to ${\cal A}(M)$ with image the 1-object full subcategory of ${\cal A}(M)$ on $F$.  Then this extends to an exact functor from ${\rm Ab}(R)$ to ${\cal A}(M)$, the image of which is an abelian subcategory which contains $F$, hence which is all of ${\cal A}(M)$.
\end{proof}

\section{Elementary duality} \label{secdual}\marginpar{secdual}

There is a duality which applies to everything that we have discussed so far.  To every definable category ${\cal D}$ there is a corresponding dual definable category ${\cal D}^{\rm d}$, defined below, which, in the case of a category ${\cal R}\mbox{-}{\rm Mod}$ of left modules, is the category ${\rm Mod}\mbox{-}{\cal R}$ of right modules.  The associated small abelian functor categories, ${\rm fun}({\cal D})$ and ${\rm fun}({\cal D}^{\rm d})$, are opposite and there is a natural bijection between definable subcategories of ${\cal D}$ and its dual (this follows most directly from the fact that a Serre subcategory of an abelian category is also a Serre subcategory of the opposite category).  The pervasiveness of duality throughout all associated structures is clear from certain equivalences of 2-categories, one of which we discuss in the next section.  The existence of this duality is clear on the 2-category, $ {\mathbb A}{\mathbb B}{\mathbb E}{\mathbb X}$, of small abelian categories with exact functors, which has a natural involution taking each abelian category $ {\mathcal A}$ to its opposite $ {\mathcal A}^{\rm op}$.

If $ {\mathcal D}={\rm Ex}({\mathcal A},{\bf Ab}) $ then the ({\bf elementary}) {\bf dual} category of $ {\mathcal D} $ is $ {\mathcal D}^{\rm d}={\rm Ex}({\mathcal A}^{\rm op},{\bf Ab}) $ and we have $ {\mathbb L}^{\rm eq+}({\mathcal D}^{\rm d})=\big({\mathbb L}^{\rm eq+}({\mathcal D})\big)^{\rm op}$, that is $ {\rm fun}({\mathcal D}^{\rm d})=({\rm fun}({\mathcal D}))^{\rm op}$. In particular, $(R\mbox{-}{\rm mod}, {\bf Ab})^{\rm fp} \simeq \big( ({\rm mod}\mbox{-}R, {\bf Ab})^{\rm fp}\big) ^{\rm op}$ (\cite[\S 7]{AusDual}, \cite[5.6]{GrJeDim}).  In the context of the model theory of modules this duality was found first for pp formulas, and termed elementary duality, in \cite{PreDual}, then extended to the category of pp-pairs (and thence to theories of modules) in \cite{HerzDual}.

\section{The bigger picture}\label{sec2cats}\marginpar{sec2cats}

What we have described above fits into a framework given by certain equivalences and anti-equivalences of 2-categories.  The 2-categories involved are categories of categories; they have categories for their objects, functors for their arrows and the ``2" refers simply to the extra structure of natural transformations between functors.  One of these 2-categories is ${\mathbb A}{\mathbb B}{\mathbb E}{\mathbb X}$:  the objects of this category are the skeletally small abelian categories and the arrows are the exact functors between these; the natural transformations between these exact functors are the 2-arrows.  Another is  ${\mathbb D}{\mathbb E}{\mathbb F}$:  its objects are the definable additive categories, its arrows are the interpretation functors and, again, natural transformations provide the ``2"-structure.  There is a contravariant equivalence between these.

\begin{theorem}\label{2cats} \cite[2.3 and comments following that]{PreRajShv} There is an anti-equivalence between $ {\mathbb A}{\mathbb B}{\mathbb E}{\mathbb X}$ and $ {\mathbb D}{\mathbb E}{\mathbb F} $:

\begin{center} $\xymatrix{{\mathbb A}{\mathbb B}{\mathbb E}{\mathbb X} \ar@{-}[rr]^{\simeq^{\rm op}} && \ar@{-}[ll]_{\simeq^{\rm op}}  {\mathbb D}{\mathbb E}{\mathbb F} }$.
\end{center}

\noindent Explicitly:

\begin{center} $\xymatrix{{\mathcal A} ={\rm fun}({\mathcal D}) =({\cal D}, {\bf Ab})^{\prod \rightarrow}  \ar@/^/[rr] & & {\mathcal D} = {\rm Ex}({\mathcal A}, {\bf Ab})   \ar@/^/[ll] }$

\noindent with the actions on morphisms being given by composition.
\end{center}

\end{theorem}

In fact, this fits into a larger picture of 2-categories, see the introduction of \cite{PreADC}

The following result is one instance of this picture.  Here ${\cal A}({\cal R})$ is the smallest abelian subcategory of ${\rm Mod}\mbox{-}{\cal R}$ which contains the finitely presented ${\cal R}$-modules, see Example \ref{flatabsex}, and ${\mathcal R}\mbox{-}{\rm Flat}$ denotes the category of flat left ${\cal R}$-modules.

\begin{prop} \label{cohexgen}\marginpar{cohexgen} \cite[7.1]{PreAxtFlat} If $ {\mathcal R} $ is any skeletally small preadditive category then $${\rm Ex}({\mathcal A}({\mathcal R}),{\bf Ab})\simeq \langle {\mathcal R}\mbox{-}{\rm Flat}\rangle .$$ If $ {\mathcal R} $ is right coherent, so $  {\mathcal A}({\mathcal R})={\rm mod}\mbox{-}{\mathcal R} $ and $ {\mathcal R}\mbox{-}{\rm Flat} $ is a definable subcategory of $ {\mathcal R}\mbox{-}{\rm Mod}$, then $${\rm Ex}({\rm mod}\mbox{-}{\mathcal R},{\bf Ab})\simeq {\mathcal R}\mbox{-}{\rm Flat}.$$
\end{prop}

The (elementary) dual (in the sense of Section \ref{secdual}) to this is the following, where ${\rm Abs}\mbox{-}{\mathcal R}$ is the category of absolutely pure right ${\cal R}$-modules (for more on these see Example \ref{flatabsex}).

\begin{prop} \label{cohexgendual}\marginpar{cohexgendual} \cite[7.2]{PreAxtFlat} If $ {\mathcal R} $ is any skeletally small preadditive category then $${\rm Ex}({\mathcal A}({\mathcal R})^{\rm op},{\bf Ab})\simeq \langle  {\rm Abs}\mbox{-}{\mathcal R}\rangle  .$$ If $ {\mathcal R} $ is right coherent, so $ {\rm Abs}\mbox{-}{\mathcal R}$ is a definable subcategory of $ {\rm Mod}\mbox{-}{\mathcal R}$, then $${\rm Ex}(({\rm mod}\mbox{-}{\mathcal R})^{\rm op},{\bf Ab})\simeq {\rm Abs}\mbox{-}{\mathcal R}.$$
\end{prop}

\section{Examples}\label{secexs}\marginpar{secexs}

Here are some examples to illustrate the definitions and results we have given.

\begin{example}\label{me}\marginpar{me} {\bf Morita equivalence:} is a type of interpretation functor.  Consider, for example, the equivalence, for any ring $R$, between $R$-modules and modules over the ring $M_n(R)$ of $n\times n$ matrices over $R$.  This takes an $R$-module $M$ to the direct sum $M^n$ regarded naturally as an $M_n(R)$-module.  The corresponding exact functor (in the other direction) between categories of pp-pairs takes the ``home sort" of $M_n(R)$ (that is, the usual, defining, sort for modules regarded as 1-sorted structures) to the $n$-th power of the home sort $(R,-)$ for $R$.  Essentially this is the identity functor since the image of that full and faithful functor is all of ${\mathbb L}_{R^{\rm op}}^{\rm eq+}$.  That is, this is really just making an alternative choice of home sort:  $(R^n,-)$ in place of $(R,-)$ in this example, $(P,-)$ where $P$ is a finitely generated projective generator of $R\mbox{-}{\rm Mod}$ in general.
\end{example}

\begin{example}\label{tilting}\marginpar{tilting} {\bf Tilting:}
Tilting is a relation between module (and other) categories much weaker than Morita equivalence.  A right module $T_R$ is a (classical 1-){\bf tilting module} if it has projective dimension $\leq 1$, if ${\rm Ext}^{1}_{R} (T,T) = 0$ and if there is an embedding of $R$ into a direct sum of copies of $T$ such that the factor module is a direct summand of a direct sum of copies of $T$.  Set $S = {\rm End}(T_{R})$; then $_{S}T$ is tilting and $R = {\rm End}(_{S}T)$.  If we impose the mild finiteness condition that $T_{R}$ and $_{S}T$ are ${\rm FP}_{2}$ (have a projective presentation with the first three terms finitely generated)  then, \cite[4.5]{PreInterp}, the following subcategories of ${\rm Mod}\mbox{-}R$:

 ${\cal F}(T) = \{ X_{R} : {\rm Hom}_{R} (T,X) = 0 \} $;

${\cal G}(T) = \{ X_{R} : {\rm Ext}^{1}_{R} (T,X) = 0 \} $

\noindent and the following subcategories of ${\rm Mod}\mbox{-}S$:

${\cal Y}(T) = \{ N_{S} : {\rm Tor}^{R}_{1} (_{S}T , N_{S}) = 0 \} $;

${\cal X}(T) = \{ N_{S}  : N_{S} \otimes T = 0 \} $

\noindent are all definable subcategories.

Moreover, under this assumption, the standard inverse pairs of tilting equivalences, ${\rm Hom}_{R} (_{S} T_{R},-):{\cal G}(T) \rightarrow {\cal Y}(T)$ with inverse $-\otimes_{S} T_{R} $, and ${\rm Ext}^{1}_{R} (_{S} T_{R},-):{\cal F}(T) \rightarrow {\cal X}(T)$ with inverse ${\rm Tor}^{R}_{1} (_{S} T_{R},-)$, are interpretation functors.  (More generally, see \cite[10.2.35]{PreNBK}, if $T$ is an $(S,R)$-bimodule then, provided $T_R$ is FP$_{n+1}$ (has a projective presentation, the first $n+2$ terms of which are finitely generated), the functor ${\rm Ext}_R^n(T,-): {\rm Mod}\mbox{-}R \rightarrow {\rm Mod}\mbox{-}S$ is an interpretation functor, and provided $_ST$ is FP$_{n+1}$ the functor ${\rm Tor}^S_n:{\rm Mod}\mbox{-}S \rightarrow {\rm Mod}\mbox{-}R$ is an interpretation functor.)
\end{example}

\begin{example}\label{KXinD4}\marginpar{KXinD4} {\bf 4-subspace representations:} Let ${\mathcal C}={\rm Mod}\mbox{-}K[T]$ and let ${\mathcal D}$ be the category of representations of the quiver $\widetilde{D_4}$ shown. $\xymatrix{1 \ar[dr] & & 2 \ar[dl] \\ & 0 \\ 4 \ar[ur] & & 3 \ar[ul]}$  Let $I:{\mathcal C} \rightarrow {\mathcal D}$ be defined on a $k[T]$-module $M$ by taking it to the representation  $
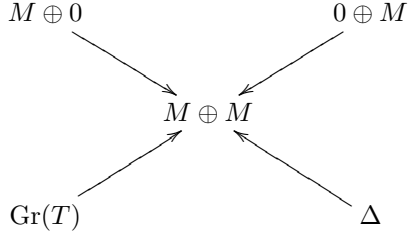
\xymatrix{M\oplus 0 \ar[dr] & & 0\oplus M \ar[dl] \\ & M\oplus M \\ {\rm Gr}(T) \ar[ur] & & \Delta \ar[ul]}$

\noindent where the maps are inclusions into the direct sum $M\oplus M$, with $\Delta =\{(m,m): m\in M\}$ the diagonal and ${\rm Gr}(T)=\{ (m, mT): m\in M\}$ the graph of multiplication-by-$T$.  The action of $I$ on morphisms is the obvious one and it is clear that this is a functor which commutes with direct products and direct limits, hence an interpretation functor.

The image of $I$ is (e.g.~using the rest of this paragraph) a definable subcategory.  Given a $\widetilde{D_4}$-representation $N$ in the image of $I$, we can recover the $k[T]$-module that it came from as follows.  For the underlying vectorspace we take $V=e_1N$ and the action is given, in words, as follows.  Let $a\in e_1N$; there is some (unique, note) $b\in e_2N$ such that $a+b \in e_4N$ (here we are identifying the $e_iN$ with their images in $e_0N$, just to simplify notation).  Then there is a unique $c\in e_1N$ such that $c+b\in e_3N$; we define $Ta=c$.  It is easy to check that $N$ is isomorphic to the image under $I$ of this $K[T]$-module.  All this can be expressed using pp formulas.  That is, we have just defined an interpretation functor $J$ going back from the image of $I$ to $K[T]\mbox{-}{\rm Mod}$ which is an inverse for $I$.

This can be found in Baur's paper \cite[p.~244]{Baur4}.
\end{example}

\begin{example}\label{flatabsex}\marginpar{flatabsex} {\bf Flat modules:} Take $R$ to be a ring (although all this works for ${\cal R}$ being a small preadditive category).  The definable category, $\langle R\mbox{-}{\rm Flat}\rangle$, generated by the left module $_RR$ includes all flat modules and equals the category of flat modules iff $R$ is right coherent.  Dually, the definable category generated by the injective right $R$-modules includes all absolutely pure (=fp-injective) modules, where a module is {\bf absolutely pure} if it is a pure submodule of every containing module; it contains exactly these modules iff the ring is right coherent. (Essentially these results go back to Eklof and Sabbagh \cite{EkSab}, \cite{SabEk}.)  These two categories are dual in the sense of Section \ref{secdual}.  Their categories of pp-imaginaries, see \ref{cohexgen} and \ref{cohexgendual}, are obtained as follows.

Consider the functor from $ {\rm Ab}(R)=(R\mbox{-}{\rm mod}, {\bf Ab})^{\rm fp} $ to $ {\rm Mod}\mbox{-}R $ which is evaluation at $R$: $ F\mapsto F(_R R) $ (the right $R$-module structure on $R$ induces the right $R$-module structure on $F(_RR)$.  We define $ {\mathcal A}(R) $ to be the image of this functor; it is, see \cite[\S6]{PreRajShv}, the smallest abelian, not necessarily full, subcategory of $ {\rm Mod}\mbox{-}R $ which contains the full subcategory of finitely presented modules.  (The objects of ${\mathcal A}({\mathcal R})$ are, \cite[6.4]{PreRajShv}, exactly the kernels of morphisms between finitely presented modules.)
\end{example}

\begin{example}\label{inj}\marginpar{inj} {\bf Injective modules:}  The category of injective right $R$-modules is a definable subcategory of ${\rm Mod}\mbox{-}R$ iff $R$ is right noetherian.  Taking $R={\mathbb Z}$ we get the category of direct sums of copies of Pr\"{u}fer modules ${\mathbb Z}_{p^\infty}$ ($p$ a nonzero prime) and ${\mathbb Q}$.  This is a definable category with no finitely presented object apart from $0$ (see \cite[18.1.1]{PreNBK}).  By the result stated above, the corresponding functor category $=$ category of pp-pairs is the category of finitely generated abelian groups.

The category of direct sums of copies of Pr\"{u}fer modules is itself a definable category (see Example \ref{finZ}) though not a definable subcategory of ${\rm Mod}\mbox{-}{\mathbb Z}$ - its Ziegler spectrum is not compact, so it needs a language with infinitely many sorts.  We will see in Example \ref{finZ} that its functor category $=$ category of pp-pairs is the category of finite abelian groups.
\end{example}

\begin{example}\label{vnreg}\marginpar{vnreg} {\bf Regular rings:} If $R$ is a von Neumann regular ring then, by \ref{cohexgen} since $R$ is coherent and every module is flat, the category of pp-pairs for $R\mbox{-}{\rm Mod}$ is ${\rm mod}\mbox{-}R$.
\end{example}

\begin{example}\label{finZ}\marginpar{finZ} {\bf Finite abelian groups:} Given a skeletally small abelian category, we may ask for which definable category is it the functor category?  For instance, take the category, ${\rm fin}\mbox{-}{\mathbb Z}$, of finite abelian groups; can we identify ${\mathcal D} ={\rm Ex}({\rm fin}\mbox{-}{\mathbb Z},{\bf Ab})$?  We will do this somewhat indirectly.

First, note that ${\rm fin}\mbox{-}{\mathbb Z}$ is a Serre subcategory of the category, ${\rm mod}\mbox{-}{\mathbb Z}$, of finitely generated abelian groups.  We will use the following result of Krause, where ${\rm Pinj}({\mathcal D})$ denotes the full subcategory on the pure-injective objects of a definable category ${\mathcal D}$.

\begin{theorem}\label{defquot}\marginpar{defquot} (\cite[5.1]{KraEx}) Suppose that ${\mathcal A}$ is a small abelian category and that ${\mathcal S}$ is a Serre subcategory.  Consider the definable categories ${\mathcal D} ={\rm Ex}({\mathcal A}, {\bf Ab})$, ${\mathcal D}'' ={\rm Ex}({\mathcal S}, {\bf Ab})$, and ${\mathcal D}' ={\rm Ex}({\mathcal A}/{\mathcal S}, {\bf Ab})$ (so ${\mathcal D}'$ is a definable subcategory of ${\mathcal D}$ and ${\mathcal D}''$ is what Krause terms a {\bf definable quotient category} of ${\mathcal D}$).

Then ${\rm Pinj}({\mathcal D}'') \simeq {\rm Pinj}({\mathcal D})/{\rm Pinj}({\mathcal D}')$ where the latter is the stable quotient category of ${\rm Pinj}({\mathcal D})$ with respect to ${\rm Pinj}({\mathcal D}')$ - that is, the objects are those of ${\rm Pinj}({\mathcal D})$ and the morphisms are those of ${\rm Pinj}({\mathcal D})$ modulo those which factor through an object of ${\rm Pinj}({\mathcal D}')$.
\end{theorem}

In fact, we will apply this to the inclusion $({\rm fin}\mbox{-}{\mathbb Z})^{\rm op} \rightarrow ({\rm mod}\mbox{-}{\mathbb Z})^{\rm op}$ of opposite categories since the computation is a little easier and \ref{defquot} gives us the whole definable category (all the modules in it being injective, in particular pure-injective).

The definable category corresponding to $({\rm mod}\mbox{-}{\mathbb Z})^{\rm op}$ is, by \ref{cohexgendual}, ${\rm Abs}\mbox{-}{\mathbb Z} = {\rm Inj}\mbox{-}{\mathbb Z}$ - the category of injective=divisible abelian groups, with ${\rm Mod}\mbox{-}{\mathbb Q}$ being its definable subcategory annihilated by $({\rm fin}\mbox{-}{\mathbb Z})^{\rm op}$.  So, by \ref{defquot}, the category ${\rm Ex}(({\rm fin}\mbox{-}{\mathbb Z})^{\rm op}, {\bf Ab})$ is the category ${\rm Inj}\mbox{-}{\mathbb Z}/{\rm Mod}\mbox{-}{\mathbb Q}$ of divisible abelian groups modulo the category of ${\mathbb Q}$-modules.  There are no non-zero morphisms from torsion modules to torsionfree ones, so this is just the category of torsion divisible abelian groups - arbitrary direct sums of copies of Pr\"{u}fer modules.

Going back to the original, opposite, categories, we have the dual definable categories and, using that a definable category is determined by its pure-injectives (\cite[3.18]{PreADC}), we conclude that ${\rm Ex}({\rm fin}\mbox{-}{\mathbb Z},{\bf Ab})$ is the category of reduced (i.e.~without nonzero divisible submodules) torsionfree abelian groups.
\end{example}

\begin{example} \label{Noriex}
{\bf Nori motives:}  Let ${\cal V}$ be a category of varieties, for example the category ${\cal V}_k$ of nonsingular projective varieties and regular maps over an algebraically closed field $k$. Roughly, the corresponding category of motives is a small abelian category built from ${\cal V}$ through which all (co)homology theories, regarded as functors from ${\cal V}$ to an abelian category, factor (essentially uniquely).  For ${\cal V}_k$ there are constructions of ``pure motives" whose properties depend on the (open) Grothendieck Standard Conjectures.  The pure motives are also expected to be the semisimple objects of the category of ``mixed motives" (corresponding to the category of all varieties over $k$); that category is expected to be recovered, as the heart of a motivic t-structure, from the category of Voevodsky ``triangulated motives".

Nori constructed a quiver $D_k$, the vertices of which are the $(X,Y,i)$ consisting of varieties $X,Y$ with $Y$ a closed subvariety of $X$, $i\in {\mathbb Z}, i\geq 0$. So this is a large quiver, in particular its path category is a ring in the broad (many-object) sense.  The arrows of $D_k$ are defined in such a way that (co)homology theories $T$ are representations of this quiver.  Nori then gave a (rather indirect) construction of an associated abelian category of motives for $T$, in particular taking $T$ to be singular homology, $(X,Y,i) \mapsto {\rm H}^i(X({\mathbb C}),Y({\mathbb C}))$ to obtain the category of Nori motives.

Barbieri-Viale suggested that there should be a topos-theoretic construction of Nori's category, Caramello saw how to do this by building the syntactic category associated to a regular theory and the details are presented in \cite{BVCL}.  In fact, the category that Nori constructs can be seen to have the universal property of the category ${\cal A}(T)$ where $T$ is the representation of $D_k$ given by singular homology, so this, and the more general $T$-motives from \cite{BV}, are examples, of a rather different kind, of what we have described, see \cite{BVP}.
\end{example}

\section{Model theory and interpretations}\label{secmodth}\marginpar{secmodth}

\noindent {\bf Many-sorted structures}  The idea is that structures may have elements of essentially different kinds or, at least, in completely disjoint parts of the structure.  Representations of quivers are natural examples:  the vector spaces corresponding to distinct vertices are such separate parts; the term `sorts' refers to these `parts'.  Even in the case of a module\footnote{In the earlier parts of this paper, I used mostly left modules but, in this section I will use {\em right} $S$-modules, that is $S^{\rm op}$-modules, to avoid too many $^{\rm op}$s appearing elsewhere in the discussion.} $M_R$ over a ring $R$, which is naturally 1-sorted, we may consider $n$-tuples of elements of $M$ as elements of a new sort - that of $n$-tuples - (one for each $n$), thus turning $M$ into a multi-sorted structure.

But why stop there?  We can add, as new sorts, pp-definable subgroups $\phi(M)$ of $M$ and of its finite powers; we can also add quotients of these by pp-definable equivalence relations.  We're restricting to pp formulas in order that all sorts inherit the additive structure, and factoring by pp-definable equivalence relations comes down to forming factor groups $\phi(M)/\psi(M)$.

There are also many definable functions connecting these sorts:  coordinate injections and projections in the case of product sorts, the canonical projection from $\phi(M)$ to $\phi(M)/\psi(M)$, but many more as we saw even in the example in Section \ref{secKT}.

The resulting (very-)many-sorted structure is denoted by $M^{\rm eq+}$.  Since `$M$' is a variable, we should consider the framework behind this.  That framework, seen earlier, is that where we take the sorts $\phi(-)/\psi(-)$ to be objects and the pp-definable maps\footnote{That is, pp-definable on {\em every} $R$-module, equivalently pp-definable on $M$ if the definable subcategory generated by $M$ is ${\rm Mod}\mbox{-}R$; if we meant just pp-definable on $M$ then the resulting category would be a Serre quotient of ${\mathbb L}_R^{\rm eq+},$ see Section \ref{secloc}.} to be arrows.  This is an abelian category which, from this point of view, is denoted ${\mathbb L}_R^{\rm eq+}$ (and which has alternative descriptions, as in \ref{equiv3}).  Then applying a module $M$ as an argument is the functor `evaluation-at-$M$' from ${\mathbb L}_R^{\rm eq+}$ to ${\bf Ab}$ and, as such, it is exact (and was denoted $\widetilde{M}$ in \ref{freeabcat}).  Equally, as we have seen, each object in ${\mathbb L}_R^{\rm eq+}$ can be regarded as a functor from ${\rm Mod}\mbox{-}R$ to ${\bf Ab}$

Now, given $M$, each one of these new sorts $N=\phi(M)/\psi(M)$ of elements is a right module over the ring $S={\rm End}_{{\mathbb L}_R^{\rm eq+}}(\phi/\psi)$ of pp-definable endomorphisms of the object $\phi/\psi$ in ${\mathbb L}_R^{\rm eq+}$.  We think of this $S$-module as ``definably contained", or ``interpreted" in $M_R$.  We could also, of course, from the $S$-module $N$ build the structure $N^{\rm eq+}$.  It might be that, within this structure, we could find a copy of the $R$-module $M$ by taking a sort of $N^{\rm eq+}$ with its definable endomorphisms.  In that case we would regard $M_R$ and $N_S$ as definably equivalent\footnote{``Interpretation-equivalent" would fit better with standard terminology in model theory, where it is customary to distinguish definability (taking only definable subsets) and interpretability (also allowing factoring by definable equivalence relations).  For our purposes, the distinction seems not worth making.} since each can be constructed/found definably within the other.  This is the relation of ``being essentially the same" that we obtain for modules if we define them to be exact functors on abelian categories.  Since we allow modules over `multi-sorted rings', that is small preadditive categories, we might use more than one sort when we build new modules inside $^{\rm eq+}$ of a module.  Also, if we just want the weaker relation of being definably contained/interpreted in (rather than definable equivalence) we don't, when interpreting new structures, have to take all the arrows between the chosen sorts in ${\mathbb L}_R^{\rm eq+}$.  That is, we can use any not-necessarily-full subcategory when interpreting new modules in old.

The module $M$ appearing above serves only to fix the definable category ${\cal D}$ (whether it be ${\rm Mod}\mbox{-}R$ or $\langle M \rangle$), so our choice of the sort $\phi/\psi$ gives us a functor from ${\cal D}$ to ${\rm Mod}\mbox{-}S$, $S={\rm End}(\phi/\psi)$, which also takes each module $M' \in {\cal D}$ to the $S$-module $\phi(M')/\psi(M')$.  This functor will commute with direct limits and direct products.  That is, it is an interpretation functor in the sense we defined earlier algebraically (Section \ref{secinterp}) but which can also be defined in this model-theoretic way.  This explains the terminology ``interpretation functor".

We can also see why, though the interpretation functor above goes from ${\cal D}$ to ${\rm Mod}\mbox{-}S$, the exact functor between the associated functor categories goes the other way.  For, we found $S$ as (a subring of) ${\rm End}_{{\rm fun}({\cal D})}(\phi/\psi)$ in ${\rm fun}({\cal D})$.  That is, we have an additive functor from the preadditive category $S$ to ${\rm fun}({\cal D})$.  By \ref{freeabcat} that extends to an exact functor from ${\rm Ab}(S^{\rm op}) = {\rm fun}({\rm Mod}\mbox{-}S)$ to ${\rm fun}({\cal D})$.  The meaning is that, once we have interpreted the action of $S$ on modules $M'\in {\cal D}$, the action of every sort of ${\mathbb L}_S^{\rm eq+}$ is thereby determined.

Of course, what we have just said about interpretation functors for modules over 1-sorted rings applies equally for modules over general small preadditive categories.

\vspace{4pt}

That pretty well wraps up what I want to say here.  The above discussion of interpretations does, however, depend on the reader having read the rest of the paper so, below, is an independent, and more general, run through the relevant ideas of many-sorted structures, languages and interpretations.

\subsection{Languages, structures and interpretations}

\noindent {\bf The idea}  of an interpretation is, roughly, that B is interpretable in A if we can find B definably within A or, to put it another way, if we can see how to manufacture B from A in a definable way. An example from undergraduate mathematics, the formation of the ring $ {\mathbb Q} $ from the ring $ {\mathbb Z}$, illustrates this quite well: we take the set, $ {\mathbb Z}^2$, of pairs of integers, at least those where the second coordinate is non-zero, and define an equivalence relation on those by $ (a,b)\sim (c,d) $ if $ad=bc $ (a ``definable" relation in the formal sense). We then factor $ {\mathbb Z}^2 $ by this relation, to get the domain of $ {\mathbb Q} $.  Then we define the addition and multiplication on $ {\mathbb Q} $ in terms of representatives of equivalence classes in the obvious (again, formally definable) way. What we did there was to produce a new kind of element: starting with the usual elements of $ {\mathbb Z} $ - the elements of the ``home sort" - we formed the sort of ordered pairs, moved to the definable subset consisting of those where the second coordinate is non-zero, and then factored by a definable equivalence relation to get a new sort which was to be the domain of the new structure. Then we equipped the new sort with a couple of functions (and we could have named the constants $ 0 $ and $ 1 $ as well), again all done definably in terms of the original structure.

In summary, we can produce new sorts by taking definable subsets of existing ones, by forming finite products of sorts and by factoring sorts by definable equivalence relations.  If we take a structure $ M $ and close it under these ways of forming new sorts, we get a collection of sorts of elements, usually denoted by $ M^{\rm eq}$. This is not just a collection of sets: there are maps connecting these sorts, for example the projections and injections between product sorts and component sorts, and also the map from a sort to its quotient by a given definable equivalence relation. In general, there are many additional definable maps between sorts and the most complete version of this structure, which we still denote $ M^{\rm eq}$, can naturally be regarded as a category, with the sorts as objects and the definable maps as morphisms.  (In that sentence we blurred the distinction between the framework - the category - and the value of this at a particular structure $M$:  for example, the construction above can be applied to any commutative domain.  Once we give the definitions, however, the difference can be clarified.)

\vspace{4pt}

\noindent {\bf The details:}
The meaning of ``definable" above is precise:  it refers to definability in some formal language appropriate for the structures involved.  Suppose that $ A $ is a structure of some, possibly many-sorted, {\bf signature}; that means a specified set of sorts (disjoint containers for elements) and a set of function, relation and constants symbols where the sorts (of each domain variable, of the value variable in the case of a function and of each constant symbol) are specified. Suppose also that $ B $ is a structure of some signature, in general totally different from that of $ A$. An {\bf interpretation} of $ B $ in $ A $ consists of the following data: for each sort $ \sigma  $ in the signature of $ B$, a sort $ \tau (\sigma ) $ in the language (built from the signature\footnote{Given a signature, we can, recursively using the function symbols, build terms from variables and any constants symbols.  Then, using equality and any other relations, and then logical connectives and quantifiers, we build formulas of the language based on that signature.  The notation for formulas like $\phi_f$, $\phi_R$ shows the free=unquantified=substitutable variables occurring in that formula.}) of $ A^{\rm eq}$; for each function symbol $ f $ in the signature of $ B $ and with domain sort $ \sigma _1\times\dots\times \sigma _n $ and value sort $ \sigma $, a formula $ \phi _f(x'_{\tau (\sigma _1)},x''_{\tau (\sigma _2)},\dots,x^{(n)}_{\tau (\sigma _n)},x_{\tau (\sigma )})$ in the language for $A$, where subscripts indicate the sorts of variables; for each relation symbol $ R $ in the signature of $ B $ and with sort $ \sigma _1\times \dots\times \sigma _n$, a formula $ \phi _R(x'_{\tau (\sigma _1)},x''_{\tau (\sigma _2)}\dots,x^{(n)}_{\tau (\sigma _n)})$; for each constant symbol $ c$, of sort $ \sigma  $ say, in the signature of $ B$, a formula $ \phi _c(x_{\tau (\sigma )})$, the formulas all in the language for $A^{\rm eq}$.\footnote{Any formula in this expanded language is equivalent to some formula written using the original signature for $A$; the new sorts for imaginaries are convenient but not, at this point, essential.}

The conditions on this data are the obvious ones: first that $ \phi _f $ should define a function from the product sort $ \tau (\sigma _1)\times \dots \times \tau (\sigma _n) $ to the sort $ \tau (\sigma ) $ and that $ \phi _c $ should define a single element; these conditions may be written as sentences in the language for $ A$; write $ T_\tau  $ for this set of sentences.  Finally we require that $B$ should be isomorphic to the structure for the signature of $ B $ which has, at sort $ \sigma  $, the set $ \tau (\sigma )(A) $ and which has, for the interpretations of the function, relation and constants symbols of that signature, the functions, relations and constants defined in $ A $ by the corresponding $ \phi $-formulas.  That is, we can definably find, or interpret, $B$ within $A$.
\vspace{4pt}

\noindent {\bf The extent of an interpretation:} This was a description of how a single structure may be interpreted in another, but it is clear that in any structure $ A' $ satisfying all the sentences in $ T_\tau $ the data of the interpretation allows one to interpret in $ A' $ some structure $ B' $ for the signature of $ B$.

\vspace{4pt}

\noindent {\bf Contravariance:} Note that the data of the interpretation goes in one direction (it is a function from something associated to $ B $ to something associated to $ A$) whereas the actual interpretation goes in the other direction (from some structures for the signature of $ A $ to some structures for that of $ B$).

\vspace{4pt}

\noindent {\bf The additive context/pp-definability:}  For more detail on the above, one may consult, for instance \cite[\S\S 5.3, 5.4]{Hod}.  In the additive context, because we want the underlying additive structure to be inherited by definable and interpreted sets, we restrict to pp formulas since these always define subgroups.  The restriction also has model-theoretic justification in the light of pp-elimination of quantifiers for modules, e.g.~see \cite[2.13]{PreBk}.\footnote{This says roughly that formulas in the language for modules over any ring reduce to simple combinations of what one can say with pp formulas.}  The notation $^{\rm eq+}$ is used in place of $^{\rm eq}$ to indicate this difference.

Also, the languages that we have used in this paper are functional - no constants other than the (definable) $ 0 $ element in each sort and no relation symbols (other than equality). The exclusion of relation symbols can be got around.  For instance if we want to look at structures consisting of a $K$-vector space and a specified subspace, then it would be natural to introduce a 1-ary relation symbol to pick out the subspace.  But an alternative is to regard such structures as representations of the quiver $A_2$ (those where the kernel of the arrow is 0 - these form a definable subcategory of the category of $K$-representations of $A_2$).  This kind of ``coding-up" of relations by introducing new sorts and function symbols can be done in general, see \cite[p.\,703]{KP1}. Not so for non-zero constants: they cannot be coded up in this way and permitting them {\em in the formal language} would move us from an additive to an affine world which is not very well-explored and where it is not at all clear that as nice a theory can be developed (e.g.~see \cite{RajPreAff}).

\vspace{4pt}

\noindent {\bf The language for ${\cal R}$-modules:} Suppose that $ {\mathcal R} $ is a small preadditive category. We set up a language $ {\mathcal L}_{\mathcal R} $ with a sort $ \sigma _P $ for each object $ P $ of $ {\mathcal R} $ and with, for each morphism $ f:P\rightarrow Q$, a function symbol, which we denote just by $ f$, from sort $ \sigma _P $ to $ \sigma _Q $.  For each sort we add a binary function symbol (for addition in that sort) and a constant symbol (for the $ 0 $ of that sort). We are interested only in additive structures so we impose the axioms which state that the structure in each sort should be that of an abelian group, then we add those which state that each named function $ f $ as above should be additive, that is, a morphism of abelian groups from the group of elements of sort $ \sigma _P $ to the group of those of sort $ \sigma _Q$.  We also add axioms expressing the atomic diagram of ${\mathcal R}$; for instance if the composition (or sum) of $f$ and $g$ is $h$ then we add a formula saying so (in terms of their actions).  An $ {\mathcal L}_{\mathcal R}$-structure $ M $ which is a model of (that is, satisfies) all these axioms therefore assigns to each object $ P $ of $ {\mathcal R} $ an abelian group $ \sigma _PM $ and to each morphism $ f:P\rightarrow Q $ of $ {\mathcal R} $ a group homomorphism from $ \sigma _PM $ to $ \sigma _QM$. In other words, $ M $ is a (covariant) additive functor from $ {\mathcal R} $ to the category $ {\bf Ab} $ of abelian groups. And conversely, every such additive functor is a model of these axioms.

\vspace{4pt}

\noindent {\bf Another language for modules:}  Suppose that ${\cal C}$ is any finitely accessible additive category with products.  That is, the finitely presented objects form a skeletally small subcategory ${\cal C}^{\rm fp}$ and every object of ${\cal C}$ is a direct limit of finitely presented objects.  For example any module category ${\cal R}\mbox{-}{\rm Mod}$ is such.  We can set up a language, with a sort for each object of (a small but not necessarily skeletal version of) ${\cal C}^{\rm fp}$, a $0$ and a symbol for addition on each sort, and a function symbol, from sort $Q$ to sort $P$ for each arrow $f:P\rightarrow Q$ in ${\cal C}^{\rm fp}$.  The reason for contravariance is that we can now regard each object $C \in {\cal C}$ as a structure for this language, setting $\sigma_PC=(P,C)$ - the group of maps from $P$ to $C$.  Since every object of ${\cal C}$ is a direct limit of objects of ${\cal C}^{\rm fp}$, every object $C$ is determined by its ``elements", that is, by the morphisms from finitely presented objects to it.  These morphisms we have split into separate sorts according their domains (for more detail see, \cite[\S 18]{PreMAMS}).

In this language, a module $M$ would be many-sorted, with a sort for each finitely presented module $A$ and the group of elements of $M$ of sort $A$ being $(A,M)$.  Of course, the usual language is the restriction of this to the single sort $(R,M)$, ``usual" {\it via} the canonical way of identifying the elements of a module with the morphisms from the free module $R$ to it.  The other sorts all lie in $^{\rm eq+}$ of this single-sorted structure but it can be useful to have them there explicitly as part of a richer language.

Take, for an example of a finitely accessible category with products, the category of torsion abelian groups:  the finite abelian groups are the finitely presented objects and product in the category is the the torsion subgroup of the usual product of groups.  So the language has a sort for each finite abelian group (taking just the indecomposables would be enough, though then the category of sorts would not be an additive category) and the ``elements" of sort, say, ${\mathbb Z}_n$, of a torsion group $D$ form the group $({\mathbb Z}_n, D)$ of elements of $D$ of order dividing $n$.

\vspace{4pt}

\noindent {\bf Yet another language for modules:}  We could, of course, use the full ``pp-imaginaries" language, based on ${\mathbb L}_{\cal R}^{\rm eq+}$, with a sort for each object of that category and a function symbol for each arrow.  This contains the above languages, with each finitely presented module $A$ being replaced by the corresponding representable functor $(A,-)$ (recall, \ref{equiv3}, that ${\mathbb L}_{\cal R}^{\rm eq+}$ is equivalent to the category of finitely presented functors on finitely presented modules).  This, furthermore, is the richest possible language which still has the same expressive power as the usual (1-sorted in the case of a ring $R$) language.  ``The same expressive power" means that everything which can be said in this richer language can be said (though at greater length) with a formula from the simpler languages.  It is richest in the sense that if we were to repeat the $^{\rm eq+}$ construction, then we would obtain a category equivalent to ${\mathbb L}_{\cal R}^{\rm eq+}$, so would not have changed the language (other than adding more copies of things we already had).  More discussion of these three kinds of languages is in \cite{PreMakVol}.

\vspace{4pt}

\noindent {\bf Definable categories and purity:}  An alternative, but equivalent, definition of purity to the one given in Section \ref{secdefcat} is that an  embedding $ f:M\rightarrow N $ of $ {\mathcal R}$-modules is pure if for every pp formula $ \phi  $ and tuple $ \overline{a} $ from $ M $ (matching the free variables of $ \phi $) we have $ \overline{a}\in  \phi (M) $ iff $ \overline{fa}\in \phi (N) $ (the direction $ \Leftarrow $ being the point). This is a weakening of the definition of elementary embedding from model theory, which is this condition for {\em any} formula (not just pp formulas).

Definable subcategories of categories of modules were first studied in the context of the model theory of modules and this was reflected in the way that they were originally described.  Each definable subcategory/subclass of ${\rm Mod}\mbox{-}R$ can be realised as the class of those modules which are direct summands of models of some complete theory $T$ of modules whose class of models is closed under finite direct sums (hence also under arbitrary direct sums and direct products).  This condition on a theory $T$ is denoted $T=T^{\aleph_0}$.  By pp-elimination of quantifiers for modules, each such class is described by closure of some set of pp-pairs (condition (iv) in \ref{chardefsub}).  My recollection is that the algebraic characterisation came from discussions of Crawley-Boevey, Krause and others; the name is from \cite{C-BTrond}.

\vspace{4pt}

\noindent {\bf Interpretation of additive structures, in summary:}  The process of interpretation can be said both in model-theoretic and algebraic terms.  Model-theoretically, we start with the language ${\cal L}$ of ${\cal R}$-modules or of some more general definable category ${\cal C}$.  We expand this language to the full pp-imaginary language $ {\mathcal L}^{\rm eq+}({\cal C})$. Then we choose a subset, $ \Sigma ' $, of the expanded set of sorts and we choose a subset of the definable functions involving only these sorts and their finite products; that gives us the new, interpreted, language.  Each of our original objects $C\in {\cal C}$ has a natural expansion, $C^{{\rm eq}+}$, to the pp-imaginaries language and this, restricted to the chosen collection of sorts and function symbols, is the structure $D=FC$ interpreted in $C$.  The class of objects obtained this way might not be closed (among ``$\Sigma'$-modules") under pure subobjects but \cite[3.8]{PreADC}, after closing in this way, we obtain a definable subcategory ${\cal D}$.

Algebraically this is said as follows.  We start with a small abelian category ${\cal A}$ and we choose a not necessarily full additive subcategory ${\cal B}$.  The interpretation takes $C \in {\rm Ex}({\cal A}, {\bf Ab})$ to its restriction to ${\cal B}$.  We may as well extend $C$ to an exact functor on the abelian subcategory of ${\cal A}$ generated by ${\cal B}$ (close ${\cal B}$ under finite direct sums, kernels and cokernels).  So there is no loss in generality in taking ${\cal B}$ to be a not necessarily full abelian subcategory of ${\cal A}$.  From each exact functor $C$ on ${\cal A}$ we get an exact functor on this ${\cal B}$; we may not get all exact functors on ${\cal B}$ this way but, after closing under pure subfunctors, we do.  Set ${\cal D} ={\rm Ex}({\cal B}, {\bf Ab})$.

Thus, either way, we get an interpretation functor $F:{\cal C} \rightarrow {\cal D}$ between definable categories, and a corresponding exact functor ${\mathbb L}^{\rm eq+}({\cal D}) = {\rm fun}({\cal D}) \rightarrow {\rm fun}({\cal C}) = {\mathbb L}^{\rm eq+}({\cal C})$ between abelian categories.

\end{document}